\documentclass[12pt]{amsart}

\input{xy}
\xyoption{all}
\usepackage{graphicx}
\usepackage{color}
\usepackage{verbatim}
\usepackage[bookmarks,colorlinks,breaklinks]{hyperref}  
\hypersetup{linkcolor=blue,citecolor=blue,filecolor=blue,urlcolor=blue} 

\newtheorem{theorem}{Theorem} [section]
\newtheorem{prop}[theorem]{Proposition}
\newtheorem{lemma}[theorem]{Lemma}
\newtheorem{cor}[theorem]{Corollary}
\newtheorem{conjecture}[theorem]{Conjecture}

\theoremstyle{definition}

\newtheorem{example}[theorem]{Example}
\newtheorem{remark}[theorem]{Remark}

\numberwithin{equation}{section}
\numberwithin{figure}{section}

\newcommand\C{{\mathbb C}}

\renewcommand\P{{\mathbb P}}

\newcommand\R{{\mathbb R}}

\newcommand\Q{{\mathbb Q}}

\newcommand\D{{\mathbb D}}

\renewcommand\phi{\varphi}

\newcommand\Gal{\operatorname{Gal}}
\newcommand\Aut{\operatorname{Aut}}

\newcommand\PSL{\mathrm{PSL}}

\newcommand\Qbar{\overline{\mathbb{Q}}}

\newcommand\Kbar{\overline{K}}

\newcommand\del{\partial} 
 
\renewcommand\>{\rangle} 
\newcommand\iso{\simeq} 
\newcommand\vol {\operatorname{vol}}

\newcommand\supp{\operatorname{supp}}   

\newcommand\Rat  {\mathrm{Rat}} 
\newcommand\M {\mathrm{M}}

\newcommand\Per {\mathrm{Per}}
\newcommand\Preper {\mathrm{Preper}}

\newcommand{\cL}{{\mathcal L}}

\usepackage{todonotes}
\usepackage{amssymb,amsmath,amsthm,amsfonts}
\usepackage{color}
\definecolor{myblue}{rgb}{0.6, 0.9, 1}
\definecolor{mygreen}{rgb}{0,0.8, 0.2}
\definecolor{purple}{rgb}{0.6,0.2,1}
\definecolor{orange}{rgb}{0.8,0,0.2}

\begin{document}

\title{Dynamics on $\P^1$:  preperiodic points and pairwise stability}

\author{Laura DeMarco and Myrto Mavraki}
\email{demarco@math.harvard.edu}
\email{mavraki@math.harvard.edu}

\date{\today}

\begin{abstract}
In \cite{DKY:quad}, it was conjectured that there is a uniform bound $B$, depending only on the degree $d$, so that any pair of holomorphic maps $f, g :\P^1\to\P^1$ with degree $d$ will either share all of their preperiodic points or have at most $B$ in common.  Here we show that this uniform bound holds for a Zariski open and dense set in the space of all pairs, $\Rat_d \times \Rat_d$, for each degree $d\geq 2$.  The proof involves a combination of arithmetic intersection theory and complex-dynamical results, especially as developed recently by Gauthier-Vigny \cite{Gauthier:Vigny}, Yuan-Zhang \cite{Yuan:Zhang:quasiprojective}, and Mavraki-Schmidt \cite{Mavraki:Schmidt}.  In addition, we present alternate proofs of the main results of DeMarco-Krieger-Ye in \cite{DKY:UMM, DKY:quad} and of Poineau in \cite{Poineau:BFT}. In fact we prove a generalization of a conjecture of Bogomolov--Fu--Tschinkel  \cite{BFT:torsion}  in a mixed setting of dynamical systems and elliptic curves. 
\end{abstract}



\maketitle

\thispagestyle{empty}

\bigskip
\section{Introduction}

Fix an integer $d \geq 2$.  Let $\P^1$ denote the complex projective line, and let $\Rat_d$ denote the space of all holomorphic maps $f: \P^1 \to \P^1$ of degree $d$.  Parameterizing these maps by their coefficients, throughout this article, we identify $\Rat_d$ with a Zariski-open subset of the complex projective space $\P^{2d+1}$.  For $f \in \Rat_d(\C)$, let $\Preper(f)$ denote its set of preperiodic points in $\P^1(\C)$.

Given any pair $f, g: \P^1 \to \P^1$ of degrees $>1$, it is known that either $\Preper(f)\cap\Preper(g)$ is a finite set or $\Preper(f) = \Preper(g)$ and the two maps have the same measure of maximal entropy \cite{BD:preperiodic, Yuan:Zhang:II}.  Moreover, if we assume that $f$ and $g$ are non-exceptional, then equality of their preperiodic points holds if and only if the maps satisfy a strong compositional relation \cite{Levin:Przytycki}.   Further background is provided in Section \ref{background}.  

In this article, we show there exists a uniform bound on the size of $\Preper(f)\cap\Preper(g)$ for general pairs $(f,g)$ of any given degree $d\geq 2$, addressing a conjecture of \cite{DKY:quad}. 
The conjecture of \cite{DKY:quad} was in part motivated by a conjecture of Bogomolov--Fu--Tschinkel in the setting of Latt\`es maps, that is, the maps arising as quotients of endomorphisms on elliptic curves \cite{BFT:torsion}. We prove a generalization of their conjecture where $f$ is  a Latt\`es map and $g$ is arbitrary. 

The proofs involve a combination of techniques from arithmetic geometry and complex dynamics.  As much as possible, we mimic the arguments for 1-parameter families of pairs $(f,g)$ acting on $\P^1\times \P^1$ of \cite{Mavraki:Schmidt}.  The new ingredients here involve the passage from one-parameter families to higher-dimensional parameter spaces:  we rely on the recent equidistribution results of Yuan and Zhang \cite{Yuan:Zhang:quasiprojective}, and we take inspiration from the analysis of bifurcation currents in \cite{Bassanelli:Berteloot}, \cite{Buff:Epstein:PCF}, \cite{Berteloot:Bianchi:Dupont}, and especially the recent work of Gauthier and Vigny \cite{Gauthier:Vigny}.

\subsection{Summary of main results}
The first main goal of this article is to prove the existence of a uniform bound on the size of the intersection $\Preper(f)\cap\Preper(g)$, for a general pair $f, g: \P^1 \to \P^1$ in each degree:

\begin{theorem}  \label{general bound}
For each degree $d \geq 2$, there is a proper closed algebraic subvariety $V_d$ of $\Rat_d\times\Rat_d$ defined over $\Qbar$ and a constant $B_d$ so that 
	$$|\Preper(f) \cap \Preper(g)| \leq B_d$$
for all pairs $(f,g) \in (\Rat_d\times\Rat_d \setminus V_d)(\C)$.   
\end{theorem}

Conjecture 1.4 of \cite{DKY:quad} posited that, for each degree $d \geq 2$, there should exist a constant $B_d$ so that every pair $f, g \in \Rat_d$ has either $\Preper(f) = \Preper(g)$ or $|\Preper(f) \cap \Preper(g)| \leq B_d$.  The proof of Theorem \ref{general bound} does not provide an explicit description of the subvariety $V_d$, and so it fails to provide a complete proof of this conjecture in any degree $d$.  However, our approach to Theorem \ref{general bound} provides an alternate proof of the main results of \cite{DKY:quad} and \cite{DKY:UMM}, stated as Theorems \ref{quad} and \ref{Lattes} below.   We also prove the following generalization of the Bogomolov--Fu-Tschinkel conjecture \cite{BFT:torsion}: 

\begin{theorem} \label{Lattes theorem intro}
For each $d\geq 2$, there exists a uniform bound $M_{d}$ so that for each elliptic curve $E$ over $\C$ and each $f\in \Rat_d(\C)$ we have that either
	$$|\Preper(f) \cap x(E_{\mathrm{tors}})| \leq M_d \quad\mbox{or} \quad \Preper(f) = x(E_{\mathrm{tors}}),$$
where $x(E_{\mathrm{tors}})$ denotes the $x$-coordinates of the torsion points of a Weierstrass model of $E$. 
\end{theorem}

The latter case, where $\Preper(f) = x(E_{\mathrm{tors}})$, will only hold if $f$ is a Latt\`es map arising as the quotient of a map on the same elliptic curve $E$; see \S\ref{preperiodic theorem} for background on the preperiodic points of Latt\`es maps.  The conjecture of Bogomolov--Fu-Tschinkel is the case of Theorem \ref{Lattes theorem intro} where $f$ is assumed to be a Latt\`es map, and it is stated below as Corollary \ref{Bft:cor}; the conjecture has recently been established by Poineau \cite{Poineau:BFT} and also follows by K\"uhne's work \cite{Kuhne:RBC} and a result of Gao-Ge-K\"uhne \cite{Gao:Ge:Kuhne}.  We also include a proof here to illustrate the differences in approach.

Easy examples (or an interpolation argument) show that there are maps $f$ and $g$ in every degree $d\geq 2$ so that 
	$$2d < |\Preper(f) \cap \Preper(g)| < \infty.$$ 
The question remains of how large this intersection can be (when finite), if we bound the degree.  Doyle and Hyde have shown there exist pairs $(f,g)$ of degree $d$ with 
	$$d^2 + 4d < |\Preper(f) \cap \Preper(g)| < \infty$$ 
for every $d\geq 2$ \cite{Doyle:Hyde}.  Corollary \ref{density cor}, stated below, shows that the general bound $B_d$ of Theorem \ref{general bound} must be at least $4d-1$.  In fact, we believe the following may hold:  

\begin{conjecture}
We can take $B_d=4d-1$ in Theorem \ref{general bound}. 
\end{conjecture}

A key ingredient in the proof of Theorem \ref{general bound} is the following result that we consider interesting in its own right:

\begin{theorem} \label{nonzero measure}
For each degree $d\geq 2$, the  pairwise-bifurcation measure $\mu_\Delta$ is nonzero on the moduli space 
	$$(\Rat_d \times \Rat_d)/\Aut \P^1$$
of all pairs $(f,g)$ of degree $d$.  
\end{theorem}

\noindent
Here, $\Aut\P^1 \iso \PSL_2\C$ acts on the space $\Rat_d\times\Rat_d$ by simultaneous conjugation, 
$$A\cdot (f,g) = (A\circ f\circ A^{-1}, \, A\circ g\circ A^{-1}),$$ 
and the moduli space is well defined as a complex orbifold of dimension $4d-1$.  The {\bf pairwise-bifurcation current} is a closed, positive $(1,1)$-current on $\Rat_d \times \Rat_d$ defined by 
\begin{equation} \label{bif current 1}
	\hat{T}_\Delta := \pi_* \left( p_1^*\hat{T} \wedge p_2^*\hat{T}\right)
\end{equation}
where $\pi: (\Rat_d \times \Rat_d)\times \P^1 \to \Rat_d \times \Rat_d$ is the projection; $\hat{T}$ is the dynamical Green current on $\Rat_d \times\P^1$ associated to the universal family; and $p_i$ is the projection $(\Rat_d \times \Rat_d)\times\P^1 \to \Rat_d\times\P^1$ given by $p_i(f_1, f_2, z) = (f_i, z)$ for $i = 1, 2$.  (See \S\ref{stable subsection} for the definition of $\hat{T}$ and \S\ref{gen bif current} for more information about $\hat{T}_\Delta$.)  The associated {\bf pairwise-bifurcation measure} is defined by 
\begin{equation} \label{bif measure 1}
	\mu_\Delta := \left(\hat{T}_\Delta\right)^{\wedge (4d-1)}
\end{equation}
on $\Rat_d\times\Rat_d$, which projects to a measure on the moduli space of pairs $(f,g)$. The current $\hat{T}_\Delta$ is a special case of the bifurcation currents introduced in \cite{Gauthier:Vigny}, extending the notions of bifurcation current and measure that characterize stability for holomorphic families of maps on $\P^1$ \cite{D:current, D:lyap, Bassanelli:Berteloot, Dujardin:Favre:critical}. With this perspective the statement of Theorem \ref{nonzero measure} may be compared with \cite[Proposition 6.3]{Bassanelli:Berteloot} that the (standard) bifurcation measure does not vanish identically on the moduli space $\M_d = \Rat_d/\Aut\P^1$. More details are given in Section \ref{nondegeneracy:bifmeasure}. 

In working towards Theorem \ref{general bound} and Theorem \ref{Lattes theorem intro}, we actually prove two contrasting results in a more general setting.  Let $S$ be a smooth and irreducible quasiprojective variety over $\C$, and let $k = \C(S)$ be its function field.   An {\bf algebraic family of pairs $(f,g)$ of degree $d\geq 2$ over $S$} is a pair of rational functions $f, g \in k(z)$ of degree $d$ for which $f$ and $g$ each induce holomorphic maps $S \times \P^1 \to \P^1$ via specialization $(t, z) \mapsto f_t(z)$.  We say that the family $(f,g)$ is {\bf isotrivial} if the induced map $S \to (\Rat_d \times\Rat_d)/\Aut\P^1$ is constant; the family $(f,g)$ is {\bf maximally non-isotrivial} if the induced map $S \to (\Rat_d\times\Rat_d)/\Aut\P^1$ is finite-to-one.  

\begin{theorem} \label{density result}
Let $S$ be a smooth and irreducible quasiprojective variety over $\C$.  Suppose that $(f,g)$ is a maximally non-isotrivial algebraic family of pairs over $S$ of degree $d\geq 2$.  Let $m = \dim_\C S$.  Then the set of points 
$$\{(s, x_1, \ldots, x_\ell) \in (S \times (\P^1)^\ell)(\C):  x_i \in \Preper(f_s) \cap \Preper(g_s) \mbox{ for each } i\}$$
is Zariski dense in $S \times (\P^1)^\ell$ for each $1 \leq \ell \leq m$.  
\end{theorem}

\begin{cor} \label{density cor}
In every degree $d \geq 2$, the set of pairs $(f,g)$ sharing at least $4d-1$ distinct preperiodic points is Zariski dense in $\Rat_d \times \Rat_d$.  
\end{cor}

Let $(f,g)$ be an algebraic family of pairs parameterized by $S$, with $m = \dim_\C S$.  An $m$-tuple ${\bf x} = (x_1, \ldots, x_m) \in (\P^1)^m$ is a {\bf rigid $m$-repeller} at $s_0 \in S$ if  (1) each $x_i$ is preperiodic to a repelling cycle for $f_{s_0}$ and for $g_{s_0}$, and (2) there is no non-constant holomorphic disk $\phi: \D \to S \times (\P^1)^m$ parametrizing a family of $m$ common preperiodic points for $(f_{\pi(\phi(t))}, g_{\pi(\phi(t))})$ with $\phi(0) = (s_0, {\bf x})$, where $\pi$ is the projection to $S$.

The following {\em non-density} result should be contrasted with the density result of Theorem \ref{density result}:  

\begin{theorem} \label{non-density result}
Fix degree $d \geq 2$, and suppose that $S$ is a smooth, irreducible quasi-projective variety,  parameterizing an algebraic family of pairs $(f,g)$ of degree $d$, defined over $\Qbar$.  Let $m = \dim_\C S$.  Suppose there exists a rigid $m$-repeller at some $s_0 \in S(\C)$, and assume that $f$ and $g$ are not both conjugate to $z^{\pm d}$ over the algebraic closure of $k = \Qbar(S)$.  Then the set of points 
$$\{(s, x_1, \ldots, x_{m+1}) \in (S \times (\P^1)^{m+1})(\Qbar):  x_i \in \Preper(f_s) \cap \Preper(g_s) \mbox{ for each } i\}$$
is {\em not} Zariski dense in $S \times (\P^1)^{m+1}$.  
\end{theorem}

Under the hypotheses of Theorem \ref{non-density result}, Theorem \ref{density result} implies that there is a Zariski-dense collection of pairs $(f_s, g_s)$ in $S(\C)$ which have at least $m$ common preperiodic points.  But we will deduce from Theorem \ref{non-density result} that -- if {\em just one} of them forms a rigid $m$-repeller at $s_0 \in S(\C)$ --  there will be a uniform bound on the number of common preperiodic points for a general pair $(f_s, g_s)$ for complex parameters $s \in S(\C)$.  See Theorem \ref{uniform bound} and Corollary \ref{our measure}.

\begin{remark} 
The family defined by $f_s(z) = z^d$ and $g_s(z) = s^{d-1} \, z^d$ for $s\in \C^*$ does not satisfy the conclusion of Theorem \ref{non-density result}, though it satisfies all the other hypotheses of the theorem.  For example, near $s_0=1$, the common fixed point at $z=1$ splits into distinct fixed points at $z=1$ for $f_s$ and $z=1/s$ for $g_s$ near $s_0$, so it forms a rigid $1$-repeller at $s_0$.  However, we have $\Preper(f_s) = \Preper(g_s)$ for all roots of unity $s$. In the terminology of \cite{Mavraki:Schmidt}, the diagonal in $\P^1\times\P^1$ is weakly $(f,g)$-special for this example, over the function field $\C(s)$.
\end{remark}

A proof of Theorem \ref{general bound} is obtained from Theorem \ref{non-density result} by showing that the pair 
	$$f_0(z) = z^d \quad\mbox{ and } \quad g_0(z) = \zeta z^d$$
for $\zeta = e^{2\pi i/(d+1)}$ has a rigid $(4d-1)$-repeller, which also implies Theorem \ref{nonzero measure}.  The rigidity for this pair $(f_0, g_0)$ is deduced from the following:

\begin{theorem} \label{rigidity of special pair}
Fix degree $d\geq 2$.  Let $f_0(z) = z^d$ and $g_0(z) = \zeta z^d$ for $\zeta = e^{2\pi i/(d+1)}$, and let $\psi = (f,g): \D \to \Rat_d\times\Rat_d$ be a holomorphic map from the unit disk $\D\subset \C$ with $\psi(0) = (f_0, g_0)$.  If 
	$$\Preper(f_t)\cap J(f_t) = \Preper(g_t) \cap J(g_t)$$
for all $t \in \D$, where $J$ denotes the Julia set, then this family of pairs is isotrivial. 
\end{theorem}

\begin{remark}
The conclusion of Theorem \ref{rigidity of special pair} is false if we allow $\zeta$ to be a root of unity of any order $\leq d$.  See \S\ref{special rigidity}.  The proof of Theorem \ref{rigidity of special pair}, and thus of Theorem \ref{nonzero measure}, does not involve any arithmetic ingredients.  This is in contrast to the proofs of related statements in \cite{Mavraki:Schmidt}.
\end{remark}

\subsection{Background and proof ideas}
Recently, there has been a series of breakthroughs in arithmetic geometry and dynamics, leading to powerful height estimates and equidistribution results, and ultimately to uniform bounds in related settings, especially for families of abelian varieties.  Dimitrov--Gao--Habegger  \cite{DGH:uniformity} and K\"uhne \cite{Kuhne:UMM} established uniformity in the Mordell-Lang conjecture following the blueprint they laid out in \cite{DGH1, Gao:Habegger, Habegger:special}.  K\"uhne \cite{Kuhne:UMM} established an arithmetic equidistribution theorem in families of abelian varieties and combined it with Gao's results \cite{Gao:generic, Gao:MixedAS} to prove uniformity in the Manin-Mumford and Bogomolov conjectures for curves in their Jacobians, a result that Yuan also obtained later with a different approach \cite{Yuan:uniform}. K\"uhne's approach has been extended to higher dimensional subvarieties of abelian varieties by Gao-Ge-K\"uhne \cite{Gao:Ge:Kuhne}; see Gao's survey article and the references therein for more about these developments \cite{Gao:survey}. 
  In a study of more general dynamical systems, the recent complex-analytic results of Gauthier-Vigny \cite{Gauthier:Vigny} and the arithmetic equidistribution theorems of Yuan-Zhang \cite{Yuan:Zhang:quasiprojective} and Gauthier \cite{Gauthier:goodheights} provide a whole host of new tools. These equidistribution results played an important role in recent work by the second author and Schmidt, who obtained for example a version of Theorem \ref{general bound} for families of pairs $(f,g)$ on $\P^1$ parameterized by curves (see Theorem \ref{MS} below) \cite{Mavraki:Schmidt}.  This article grew out of an effort to synthesize these ideas and to extend the results of the first author with Krieger and Ye in \cite{DKY:UMM, DKY:quad}.

For maps $f, g: \P^1\to\P^1$ defined over $\C$, uniform bounds on the intersections $\Preper(f)\cap\Preper(g)$ were obtained by DeMarco-Krieger-Ye for two families of maps, with proofs that also involved both arithmetic and complex-dynamical techniques, but which relied on computations specific to the families studied:

\begin{theorem}  \cite{DKY:quad}  \label{quad}
Suppose that $f_t$ is the family of quadratic polynomials, $f_t(z) = z^2 +t$ for $t \in \C$.  There is a uniform $B$ so that 
	$$|\Preper(f_{t_1}) \cap \Preper(f_{t_2})| \leq B$$
for all $t_1 \not= t_2$.
\end{theorem}

\begin{theorem}  \cite{DKY:UMM} \label{Lattes}
Suppose that $f_t$ is the family of flexible Latt\`es maps $f_t(z) = (z^2-t)^2/(4z(z-1)(z-t))$ for $t \in \C\setminus\{0,1\}$.  There is a uniform $B$ so that 
	$$|\Preper(f_{t_1}) \cap \Preper(f_{t_2})| \leq B$$
for all $t_1 \not= t_2$.
\end{theorem}

\noindent
Theorem \ref{Lattes} addressed a conjecture of Bogomolov, Fu, and Tschinkel \cite{BFT:torsion} about torsion points on pairs of elliptic curves.  Indeed, the preperiodic points of the Latt\`es map $f_t$ are the images of the torsion points on the Legendre curve $\{y^2 = x(x-1)(x-t)\}$ under the projection to the $x$-coordinate. The full conjecture from \cite{BFT:torsion} is now a theorem:  there exists a uniform bound $B$ so that every pair $(f,g)$ of (flexible) Latt\`es maps (in any choice of coordinates on $\P^1$) will either share all of their preperiodic points or have at most $B$ in common.  A proof was recently obtained by Poineau in \cite{Poineau:BFT}, and it can also be deduced from the ``relative Bogomolov" theorem proved by K\"uhne \cite{Kuhne:RBC} or the main theorem of \cite{Gao:Ge:Kuhne}.  We present an alternate proof as a corollary to Theorem \ref{Lattes theorem intro}.  

Note that Theorems \ref{quad} and \ref{Lattes} each covered a {\em two}-parameter family of pairs $(t_1, t_2)$, but, as mentioned above, the proofs were specific to these particular families.  For {\em one}-parameter families of pairs, Mavraki-Schmidt recently obtained uniform bounds that did not rely on particular dynamical features of the maps:

\begin{theorem} \cite{Mavraki:Schmidt} \label{MS}
Let $C$ be any algebraic curve in $\Rat_d \times \Rat_d$ defined over $\Qbar$, parameterizing a pair of maps $(f_t, g_t)$ for $t \in C(\C)$.  Then there exists a constant $B = B(C)$ so that, for each $t \in C(\C)$, either 
	$$|\Preper(f_t) \cap \Preper(g_t)| \leq B \qquad \mbox{or} \qquad \Preper(f_t) = \Preper(g_t).$$
\end{theorem}

In this article, we follow the strategy of \cite{Mavraki:Schmidt} to treat more general families.  Their approach uses an arithmetic equidistribution theorem of Yuan and Zhang \cite{Yuan:Zhang:quasiprojective} amongst other ingredients, allowing them to reduce the problem to a result of Levin-Przytycki on maps sharing a measure of maximal entropy \cite{Levin:Przytycki}. But in order to apply Yuan-Zhang's equidistribution theorem, the challenge is to prove that a {\em non-degeneracy} hypothesis is satisfied.  This non-degeneracy is defined in terms of a volume of a certain adelically metrized line bundle, and its positivity can be deduced from showing that a certain measure is nonzero \cite[Lemma 5.4.4]{Yuan:Zhang:quasiprojective}. In \cite{Mavraki:Schmidt} the authors relied on arithmetic ingredients to provide a characterization of this positivity for $1$-parameter families of products $(f,g)$ acting on $\P^1 \times \P^1$ \cite[Theorems 4.1 and 4.3]{Mavraki:Schmidt}. In contrast, to prove that this positivity condition is satisfied over the full space of pairs $\Rat_d\times\Rat_d$, we interpret it here as a notion of {\em dynamical stability}, inspired by the recent work of Gauthier and Vigny \cite{Gauthier:Vigny}. The non-degeneracy condition is then reduced to showing the positivity of the bifurcation measure $\mu_\Delta$ defined above in \eqref{bif measure 1}. We exhibit this positivity by mimicking the proofs that Misieurewicz maps lie in support of the (usual) bifurcation measure in the moduli space of maps of degree $d$ on $\P^1$ \cite{Buff:Epstein:PCF}, and the proofs of \cite[Proposition 3.7]{Berteloot:Bianchi:Dupont} and \cite[Lemma 4.8]{Gauthier:Vigny}.  Finally, the rigidity we rely upon to prove Theorem \ref{general bound} (stated as Theorem \ref{rigidity of special pair}), follows from a general treatment of monomial maps and symmetries of maps on $\P^1$. The analogous rigidity result we rely upon for Theorem \ref{Lattes theorem intro} follows by repeated applications of the main theorem in \cite{D:stableheight} on the stability of a family of maps $f$ on $\P^1$ equipped with a marked point.

\subsection{Outline} 
In Section \ref{background}, we provide some background on the dynamics of maps on $\P^1$, recalling the notion of exceptional map and the relation between the measures of maximal entropy and the preperiodic points of the map.  In Section \ref{likely intersection} we prove Theorem \ref{density result} and deduce Corollary \ref{density cor}. These results are presented as a consequence of the main theorem in \cite{D:stableheight}.   
For a proof of Theorem \ref{non-density result}, we follow the strategy of Mavraki-Schmidt in \cite{Mavraki:Schmidt}. To prove that the non-degeneracy assumption required to use the equidistribution result in \cite{Yuan:Zhang:quasiprojective} is satisfied, we use the complex-analytic tools -- of bifurcation currents and measures -- developed by Gauthier and Vigny in \cite{Gauthier:Vigny}. More precisely, in Section \ref{nondegeneracy:bifmeasure} we recall the definitions from \cite{Gauthier:Vigny} of a generalized bifurcation current associated to an arbitrary family of polarized dynamical systems and subvarieties. 
 In Proposition \ref{transverse} we show that certain rigid pre-repelling parameters are in the support of our (generalized) bifurcation measure, giving a criterion for non-degeneracy. 
 We complete the proof of Theorem \ref{non-density result} in Section \ref{unlikely intersection} and deduce its consequence towards uniform bounds in the number of common preperiodic points therein; see Theorem \ref{uniform bound}. 
In Section \ref{polynomials}, we include an explanation of how this method gives an alternative proof of Theorem \ref{quad}, using the special quadratic polynomial examples of \cite{Doyle:Hyde} and the results of Mavraki-Schmidt \cite{Mavraki:Schmidt}.  
In Section \ref{monomials}, we prove Theorem \ref{rigidity of special pair}, and we construct a rigid repeller, completing the proof of Theorem \ref{general bound}. 
In Section \ref{BFT section}, we study pairs $(f,g)$ where $f$ is a Latt\`es map and prove Theorem \ref{Lattes theorem intro}. 

\subsection{Acknowledgements}  We would like to thank Thomas Gauthier, Trevor Hyde, Harry Schmidt, and Gabriel Vigny for many helpful discussions during the preparation of this article.  The authors were supported by grants DMS-2050037 and DMS-2200981 from the National Science Foundation.

\bigskip
\section{Background on 1-dimensional dynamics}
\label{background}

In this section, we provide some important background information on the dynamics of maps $f: \P^1 \to \P^1$ defined over $\C$.  

\subsection{The measure of maximal entropy}  \label{mme}  For each rational map $f: \P^1\to \P^1$ of degree $d \geq 2$, there is a unique probability measure $\mu_f$ of maximal entropy.  Its support is equal to the Julia set of $f$, and it is characterized by the properties that it has no atoms, so $\mu_f(\{z\}) = 0$ for all $z \in \P^1(\C)$, and $\frac{1}{d}f^* \mu_f = \mu_f$, meaning that
	$$\frac{1}{d} \int_{\P^1} \left(\sum_{f(x)=y} \phi(x) \right) \, \mu_f(y)  = \int_{\P^1} \phi(x) \, \mu_f(x)$$
for all continuous functions $\phi$ on $\P^1$  \cite{Mane:unique, FLM, Lyubich:entropy}.

\subsection{Exceptional maps}
We say that a map $f : \P^1\to \P^1$ of degree $d\geq 2$ is {\bf exceptional} if it is the quotient of an affine transformation of $\C$; see \cite{Milnor:Lattes} for details.  Every exceptional $f$ is conjugate by an element of $\Aut \P^1 \iso \PSL_2\C$ to a power map $z^{\pm d}$, a Tchebyshev polynomial $\pm T_d$, or it is a {\bf Latt\`es map}, meaning that it is the quotient of a map on an elliptic curve. The exceptional maps are distinguished by properties of their measures $\mu_f$.  In each case, the Julia set $J(f)$ is a real submanifold of $\P^1(\C)$ (with boundary, in the case of the Tchebyshev polynomials), and the measure $\mu_f$ is absolutely continuous with respect to the Hausdorff measure on $J(f)$.  Zdunik proved the converse:  the exceptional maps are the only maps for which this absolute continuity can hold \cite{Zdunik}.  

\subsection{Preperiodic points and the maximal measure}  \label{preperiodic theorem}
It is well known that the preperiodic points of $f$ are uniformly distributed with respect to the measure $\mu_f$.  That is, defining discrete measures
	$$\mu_{n,m} = \frac{1}{d^n} \sum_{f^n(z)=f^m(z)}  \delta_z$$
in $\P^1$ for every pair of integers $n > m\geq 0$, then for any sequence $(n_k,m_k)$ of integers $n_k > m_k\geq 0$ with $\max\{n_k,m_k\} \to \infty$ as $k\to \infty$, the measures $\mu_{n_k, m_k}$ converge weakly to the measure $\mu_f$ \cite{Mane:unique, FLM, Lyubich:entropy}.  

But the preperiodic points determine the measure $\mu_f$ in a stronger sense, without ordering them by period or orbit length:

\begin{theorem} \cite{Levin:Przytycki, BD:preperiodic, Yuan:Zhang:II} \label{LPYZ} 
For any maps $f, g: \P^1\to \P^1$ of degrees $>1$ defined over $\C$, the following are equivalent:
\begin{enumerate} 
\item  $|\Preper(f) \cap \Preper(g)| = \infty$; and 
\item  $\Preper(f) = \Preper(g)$;
\end{enumerate}
and these conditions imply that 
\begin{enumerate}
\item[(3)]  $\mu_f = \mu_g$.
\end{enumerate}
Moreover, if at least one of $f$ or $g$ is not conjugate to a power map, then (3) is equivalent to (1) and (2).
\end{theorem}

\begin{remark} \label{LP remark}
The main theorem of \cite{Levin:Przytycki} states that if $\mu_f = \mu_g$, and if $f$ and $g$ are non-exceptional, then there exist iterates $f^n$ and $g^m$ and positive integers $\ell$ and $k$ so that 
	$$(g^{-m}\circ g^{m}) \circ g^{\ell m} = (f^{-n}\circ f^{n}) \circ f^{k n}$$
for some (possibly multi-valued) branches of the inverse $g^{-m}$ and $f^{-n}$.  It follows that $\Preper(f) = \Preper(g)$; see \cite[Theorem A and Remark 2]{Levin:Przytycki}.  
\end{remark}

As the equivalences of Theorem \ref{LPYZ} are not stated this way in the literature, we outline the proof ingredients.

\begin{proof}[Sketch proof of Theorem \ref{LPYZ}]
The implication that (2) $\implies$ (1) is immediate, because all maps of degree $>1$ have infinitely many distinct preperiodic points.  The implications (1) $\implies$ (2) $\implies$ (3) are proved in \cite[Theorem 1.2]{BD:preperiodic} and \cite[Theorems 1.3 and 1.4]{Yuan:Zhang:II}.  The key input is the equidistribution of points of small canonical height for a map $f: \P^1\to \P^1$, working over number fields or, more generally, fields that are finitely generated over $\Q$.  

If $f$ and $g$ are non-exceptional, then the implication (3) $\implies$ (2) is proved in \cite[Theorem A and Remark 2]{Levin:Przytycki}.  See Remark \ref{LP remark}.

It remains to carry out a case-by-case analysis of the measures for exceptional maps, appealing to Zdunik's characterization of exceptional maps by their measures in \cite{Zdunik}.

If $f$ is a Tchebyshev polynomial $\pm T_d$, its measure $\mu_f$ is supported on a closed interval.  If $\mu_g = \mu_f$, then $g$ must be equal to $\pm T_e$, where $e = \deg g$; indeed, we know that $g$ must be conjugate to $\pm T_e$, but the only $A \in \Aut \P^1$ for which $A_*\mu_g = \mu_g$ are $A(z) = \pm z$.  All of these maps have the same sets of preperiodic points.

If $f(z) = z^{\pm d}$, then $\mu_f$ is the uniform distribution on the unit circle, and $\mu_f = \mu_g$ implies that $g$ must also be a power map.  In this case, either $\Preper(f) = \Preper(g)$ or $\Preper(f) \cap \Preper(g) = \emptyset$; writing $g(z) = \alpha z^{\pm e}$ with $|\alpha|=1$ and depending on whether or not $\alpha$ is a root of unity.

Finally, suppose $f$ is a Latt\`es map.  Then the measure $\mu_f$ is the projection of the Haar measure from the associated elliptic curve.  In particular, the measure knows the branch points of this quotient map and the ramification degree at each point.  In other words, the measure $\mu_f$ uniquely determines the orbifold structure on the quotient Riemann sphere.  As such, it uniquely determines the isomorphism class of the elliptic curve over $\C$ and thus the set of preperiodic points of $f$, which is equal to the projection of the torsion points of the elliptic curve.  It follows that if $\mu_f = \mu_g$ for some $g$, then $g$ is a Latt\`es map from the same elliptic curve with the same set of preperiodic points.  See \cite{Milnor:Lattes} for more information on these Latt\`es maps.
\end{proof}

\begin{remark}
In \cite{Pakovich:general}, Pakovich proved that each $f \in \Rat_d$ of degree $d \geq 4$ with $2d-2$ distinct critical values satisfies 
	$$\{g \in \Rat_d: \mu_g = \mu_f\} = \{f\}.$$ 
As these form a Zariski dense and open subset of $\Rat_d$, this implies, when combined with Theorem \ref{LPYZ} for degrees $d \geq 4$, the set of pairs $(f,g)$ with $|\Preper(f) \cap \Preper(g)| = \infty$ is {\em not} Zariski dense in the space $\Rat_d\times\Rat_d$ of all pairs with $d\geq 4$.  
\end{remark}

\subsection{Why not periodic points?}  
It is reasonable to ask why we work with all preperiodic points and not the subset $\Per(f)$ of periodic points of $f$, which are also uniformly distributed with respect to $\mu_f$.  Of course, the uniform bound on $|\Preper(f)\cap \Preper(g)|$ in Theorem \ref{general bound} is stronger than a bound on $|\Per(f)\cap \Per(g)|$, but there are two underlying reasons for our focus on preperiodic points.  First, if we were to replace $\Preper(f)$ with $\Per(f)$, then the equivalences of Theorem \ref{LPYZ} would break down.  For example, if $f$ and $g$ are Latt\`es maps induced from $P\mapsto 2P$ and $P \mapsto 3P$ on the same elliptic curve via the same projection to $\P^1$, then $\mu_f = \mu_g$ with $|\Per(f) \cap \Per(g)| = \infty$ but $\Per(f) \not= \Per(f)$. 
 Also, there are many non-exceptional examples with $\mu_f = \mu_g$ but $|\Per(f) \cap \Per(g)| < \infty$, such as $f(z) = z^2 + c$ and $g(z) = -(z^2+c)$ for $c \in \C$.  On the other hand, it follows from the proof of \cite[Theorem A]{Levin:Przytycki} that if $\mu_f = \mu_g$ for non-exceptional $f$ and $g$, and if there exists just one common {\em repelling} periodic point $x \in \Per(f) \cap \Per(g)$, then $\Per(f) = \Per(g)$; see \cite[Theorem 1.5]{Ye:symmetries}.  A second reason is our method of proof and original motivation for this project.  For maps $f: \P^1\to \P^1$ defined over $\Qbar$, the preperiodic points for $f$ are precisely the points in $\P^1(\Qbar)$ for which the canonical height $\hat{h}_f$ vanishes \cite{Call:Silverman}.  Though somewhat hidden in this article, much of our analysis is, fundamentally, about properties of these height functions.

\subsection{Stability and the bifurcation current}  \label{stable subsection}
Let $\Lambda$ be a connected complex manifold and $f: \Lambda \times\P^1 \to \Lambda \times \P^1$ a holomorphic map defined by $(\lambda, z)\mapsto (\lambda, f_\lambda(z))$ where each $f_\lambda$ has degree $d \geq 2$.  The measures $\mu_{f_\lambda}$ can be packaged together into a positive $(1,1)$-current on the total space $\Lambda \times\P^1$ as follows.  Let $\omega$ be a smooth and positive $(1,1)$-form on $\P^1$ with $\int_{\P^1} \omega = 1$, and consider $p^*\omega$ on $\Lambda \times\P^1$, where $p: \Lambda \times \P^1 \to \P^1$ is the projection.  The {\bf dynamical Green current} for $f$ on $\Lambda\times\P^1$ is
\begin{equation} \label{Green}
	\hat{T}_f = \lim_{n\to\infty} \frac{1}{d^n} (f^n)^* (p^*\omega).
\end{equation}
Then $\hat{T}_f$ is a closed, positive $(1,1)$-current on $\Lambda \times \P^1$ with continuous potentials, and the slice current $\hat{T}_f|_{\{\lambda\}\times\P^1}$ coincides with the measure $\mu_{f_\lambda}$.   See, for example, \cite[\S3]{Dujardin:Favre:critical}.  If $\Lambda = \Rat_d$ is the space of all maps of degree $d$, then we simply denote this current by $\hat{T}$, as in the Introduction.

We say the family $f_\lambda$, for $\lambda \in \Lambda$, is {\bf stable} if the Julia sets of $f_\lambda$ are moving holomorphically with $\lambda$; see \cite{Mane:Sad:Sullivan, Lyubich:stability, McMullen:CDR} for background.  Following \cite{D:current, Dujardin:Favre:critical}, the {\bf bifurcation current} for $f$ is defined by 
\begin{equation} \label{T bif}
	T_{f, \mathrm{bif}} = (\pi_\Lambda)_* \left( \hat{T}_f \wedge [C] \right)
\end{equation}
where $\pi_\Lambda: \Lambda \times \P^1 \to \Lambda$ is the projection and $[C]$ is the current of integration along the critical locus of $f$; it is a closed and positive $(1,1)$-current on $\Lambda$ with continuous potentials.  In \cite{D:current}, it is proved that $T_{f, \mathrm{bif}} = 0$ if and only if the family is stable.

For {\em algebraic} families, namely where $\Lambda$ is a smooth quasiprojective complex algebraic variety, and $f$ is a morphism, McMullen proved in \cite{McMullen:families} that the family $\{f_\lambda: \lambda\in \Lambda\}$ is stable if and only if either $f$ is isotrivial or $f$ is a family of flexible Latt\`es maps.  (The family $f$ is {\bf isotrivial} if all $f_\lambda$ are conjugate by elements of $\Aut\P^1$.)  Specifically, McMullen proved that if $f$ is stable but not isotrivial, then each critical point of $f_\lambda$ is preperiodic for all $\lambda \in \Lambda$.  In \cite{Dujardin:Favre:critical}, Dujardin and Favre extended this result by studying the iterates of each critical point independently.  Namely, if $c: \Lambda \to \P^1$ parameterizes a critical point for $f_\lambda$, they introduced the current 
\begin{equation} \label{one cp}
	\hat{T}_{f,c} := (\pi_\Lambda)_* \left(\hat{T}_f \wedge [\Gamma_c]\right)
\end{equation}
on $\Lambda$, where $\Gamma_c \subset \Lambda \times \P^1$ is the graph of $c$, and they proved that $\hat{T}_{f,c} = 0$ if and only if $f$ is isotrivial or $c$ is persistently preperiodic for $f_\lambda$ \cite[Theorems 2.5 and 3.2]{Dujardin:Favre:critical}.  

\subsection{Stability of a marked point}
Assume that $\Lambda$ is a smooth quasiprojective complex algebraic variety.  Suppose that $a \in \P^1(k)$ is any point defined over the function field $k = \C(\Lambda)$ defining a holomorphic map $a: \Lambda \to\P^1$.  The pair $(f,a)$ is {\bf isotrivial} if both $f$ and $a$, after changing coordinates by M\"obius transformation defined over a finite extension of $k = \C(\Lambda)$, become independent of the parameter $\lambda$.  In other words, working over $\C$, the group $\Aut \P^1$ acts on pairs $(f,a) \in \Rat_d \times\P^1$ by $A\cdot (f,a) = (A\circ f \circ A^{-1}, A\circ a)$, and a pair $(f,a)$ defined over $k$ is isotrivial if the associated map $\Lambda \to (\Rat_d\times\P^1)/\Aut\P^1$ is constant.  

Similarly to \eqref{one cp}, we can define $\hat{T}_{f,a} := (\pi_\Lambda)_* \left(\hat{T}_f \wedge [\Gamma_a]\right)$ by intersecting the graph $\Gamma_a$ with $\hat{T}_f$ in $\Lambda\times \P^1$.   Then $\hat{T}_{f,a} = 0$ if and only if the pair $(f,a)$ is either isotrivial or persistently preperiodic \cite[Theorem 1.4]{D:stableheight}.  (Strictly speaking, the theorem there is only proved for $\Lambda$ of dimension 1, but it holds more generally, and it is not formulated in terms of the current $\hat{T}_{f,a}$.  The equivalence between the stability condition there and the vanishing of $\hat{T}_{f,a}$ is proved in \cite[Theorem 9.1]{D:lyap}.)  This characterization of stability was reproved and extended to a more general setting by Gauthier and Vigny in \cite{Gauthier:Vigny}.  We will need the following consequence:  

\begin{theorem}  \label{stable point} \cite{D:stableheight}  Suppose $\Lambda$ is a smooth, irreducible complex quasiprojective algebraic variety, and let $k = \C(\Lambda)$ be its function field.  Suppose that $f \in k(z)$ defines a holomorphic family of maps $f_\lambda: \P^1\to \P^1$ of degree $d\geq 2$ for $\lambda \in \Lambda$; fix $a \in \P^1(k)$ defining a holomorphic map $a: \Lambda \to \P^1$.  If the pair $(f,a)$ is neither isotrivial nor persistently preperiodic, then there exists $\lambda_0 \in \Lambda$ for which $a(\lambda_0)$ is preperiodic to a repelling cycle for $f_{\lambda_0}$.
\end{theorem}

\begin{proof}
The hypothesis that $(f,a)$ is neither isotrivial nor persistently preperiodic on $\Lambda$ implies we can find an algebraic curve $C$ in $\Lambda$ along which $(f,a)$ is neither isotrivial nor persistently preperiodic.  Indeed, if $(f,a)$ is not isotrivial, then the associated map $\Lambda \to (\Rat_d \times\P^1)/\Aut\P^1$ is non-constant; by the irreducibility of $\Lambda$, every $\lambda \in \Lambda$ is contained in some algebraic curve along which $(f,a)$ is not isotrivial. But if the pair is persistently preperiodic along all such curves, then the pair would be preperiodic on all of $\Lambda$.  Note, moreover, that if $(f,a)$ is neither isotrivial nor persistently preperiodic on a curve $C$, then it is also the case on the complement of any finite set of points in $C$.  So it suffices to prove the result for a smooth and quasiprojective curve $\Lambda$.     

Now assume that $\Lambda$ is a smooth, quasiprojective algebraic curve defined over $\C$.  From \cite[Theorem 1.4]{D:stableheight}, the hypothesis that $(f,a)$ is neither isotrivial nor persistently preperiodic implies that the sequence of holomorphic functions $\{\lambda \mapsto f_\lambda^n(a(\lambda))\}_{n\geq 0}$ fails to be normal on $\Lambda$.  Thus, as a consequence of Montel's theorem, there must be a parameter $\lambda_0 \in U$ and positive integer $n_0$ so that $f_{\lambda_0}^{n_0}(a(\lambda_0))$ is a repelling periodic point for $f_{\lambda_0}$; see \cite[Proposition 5.1]{D:stableheight}.  \end{proof}

\bigskip
\section{Zariski density of preperiodic points}\label{likely intersection}
\label{density}

In this section we prove Theorem \ref{density result}, restated here as Theorem \ref{Zariski dense}.

Throughout, we assume that $S$ is a smooth and irreducible quasiprojective variety over $\C$. Let $k = \C(S)$ be its function field.   An {\bf algebraic family of pairs} $(f,g)$ over $S$ is a pair of rational functions $f, g \in k(z)$ for which $f$ and $g$ each induce holomorphic maps $S \times \P^1 \to \P^1$ via specialization $(s, z) \mapsto f_s(z)$.  The pair $(f,g)$ induces a holomorphic map we denote by 
	$$\Phi = (f,g) :  S \times (\P^1\times\P^1) \to S \times (\P^1\times\P^1)$$
given by $(s, x,y) \mapsto (s, f_s(x), g_s(y))$.  We say that $\Phi = (f,g)$ has {\bf degree} $d \geq 2$ if $f_s$ and $g_s$ are both of degree $d$ for each $t\in S$.   

Recall that $f \in k(z)$ is {\bf isotrivial} if the induced map $S \to \Rat_d$ given by $s \mapsto f_s$ has constant image in the quotient space $\M_d = \Rat_d/\Aut\P^1$.  Equivalently, there exists a finite extension $k'$ of $k$ and a linear fractional transformation $B$ defined over $k'$ so that $B \circ f \circ B^{-1}$ is in $\C(z)$.   An algebraic family of pairs $(f,g)$ is {\bf isotrivial} if the induced map $S \to \Rat_d \times \Rat_d$ is constant when passing to the quotient space $(\Rat_d\times\Rat_d)/\Aut \P^1$.  Here, $\Aut\P^1 \iso \PSL_2\C$ is acting diagonally by conjugation, so that $A\cdot (f,g) = (A\circ f\circ A^{-1}, A\circ g\circ A^{-1})$ and $\dim \; (\Rat_d\times\Rat_d)/\Aut\P^1 = 4d-1$.  We say that an algebraic family of pairs $\Phi = (f,g)$ over $S$ is {\bf maximally nonisotrivial} if the family determines a {\em finite} map from $S$ to $(\Rat_d\times\Rat_d)/\Aut\P^1$.  

For each integer $\ell\geq 1$, we let $\Phi^{(\ell)}$ denote the map on $S \times (\P^1\times\P^1)^\ell$ given by the product action of $\Phi$ on the fiber power.  

\begin{theorem} \label{Zariski dense}
Suppose $\Phi$ is a maximally nonisotrivial algebraic family of pairs over $S$, of degree $d \geq 2$, and let $\Delta \subset \P^1\times\P^1$ be the diagonal.  The preperiodic points of $\Phi^{(\ell)}$ in $S\times \Delta^{\ell}$ form a Zariski dense subset of $S \times \Delta^\ell$, for every $1 \leq \ell \leq \dim S$. 
\end{theorem}

\noindent
The key ingredient in the proof is the characterization of stability of marked points $(f,a)$ -- for an algebraic family of maps $f$ on $\P^1$ and holomorphic $a: S\to \P^1$ -- from \cite{D:stableheight}; see Theorem \ref{stable point}.

\subsection{Proof of Theorem \ref{Zariski dense} when dimension of $S$ is 1}  \label{dim 1} For simplicity, we first present the proof when $S$ is a curve.  

Let $\Omega$ be any Zariski-open subset of $S \times \P^1$.  It suffices to show there exists a single point $(s_0, z_0) \in \Omega$ for which $z_0 \in \Preper(f_{s_0}) \cap \Preper(g_{s_0})$.  Choose any irreducible, algebraic curve $P \subset S\times\P^1$ parameterizing a periodic point of $f$ that intersects $\Omega$.  Note that there are infinitely many choices of such curves, because $f_s$ has infinitely many periodic points for every $s \in S$.  In fact, there exists a choice of $s$ so that all periodic points of sufficiently large period for $f_s$ will lie in $\Omega$ (a fact that will be relevant to our argument). 
We may view the curve $P$ as the graph of a point in $\P^1(k')$ for some finite extension $k'$ of $k = \C(S)$.  Let $S'\to S$ denote a finite branched covering map so that $k' = \C(S')$.  By construction, the pair $(f,P)$ is persistently preperiodic over $S'$.

Now assume that the pair $(g, P)$ is not isotrivial.  If the pair $(g,P)$ is also persistently preperiodic for $g$, then we are done.  Otherwise, by Theorem \ref{stable point}, there exists $s_0\in S'$ at which $P_{s_0}$ is preperiodic for $g_{s_0}$.  This completes the proof under this assumption of non-isotriviality of $(g,P)$.

If $(g,P)$ is isotrivial, it is convenient to pass to a further finite branched cover $S' \to S$, if necessary, and change coordinates so that the family $g_s$ is independent of $s \in S'$ and the point $P_s$ is constant.  In these new coordinates, if the pair $(g,P)$ is persistently preperiodic, the proof is complete.  If the pair $(g,P)$ is isotrivial but with infinite orbit, we then repeat the argument with a different choice of curve $P$.  If the pair $(g,P)$ is isotrivial for all curves $P$ parameterizing points of a given large period for $f$, then each of these periodic points for $f$ is constant in these coordinates over $S'$.  In particular, by an interpolation argument, $f$ itself must be a constant family.  More precisely, choosing any period $N > 2d+1$, we would be able to find a set of distinct points $z_1, z_2, \ldots, z_N \in \P^1(\C)$ so that $f_s(z_i) = z_{i+1}$ for all $s \in S'$ and all $i = 1, \ldots, N-1$, and this would imply that the maps $f_s$ are constant in $s$ \cite[Lemma 2.5]{D:stableheight}.   In other words, the pair $(f,g)$ is isotrivial, violating the hypothesis.  This completes the proof.  

\subsection{Proof of Theorem \ref{Zariski dense} for any base $S$}
Let $m = \dim_\C S$.  Let $\Omega$ be any Zariski-open subset of $S \times (\P^1)^m$.  It suffices to show that there exists a single point $(s_0, z_1, z_2, \ldots, z_m) \in \Omega$ for which 
	$$z_i \in \Preper(f_{s_0})\cap \Preper(g_{s_0})$$
for all $i = 1, \ldots, m$.   Indeed, this shows Zariski density of the preperiodic points of $\Phi^{(m)}$ in $S \times \Delta^m$.  For $S \times \Delta^\ell$ with $\ell < m$, we observe that the projection, forgetting some factors of $\Delta$, will still be Zariski dense, which will complete the proof.  

The proof proceeds by induction on the dimension of $S$.

First let $P_1$ denote an irreducible subvariety in $S \times (\P^1)^m$ of codimension 1 having nontrivial intersection with $\Omega$, and for which $z_1$ is periodic for $f_s$ for all $(s, z_1, \ldots, z_m) \in P_1$.  As the periodic points of $f$ are Zariski dense in $S \times \P^1$, we can always find such a $P_1$.  Projecting $P_1$ to the $z_1$ coordinate, we may view this $P_1$ as single marked point defined on a finite (branched) cover $S'$ of $S$.  Now consider the pair $(g,P_1)$ over $S'$, abusing the notation slightly to identify $P_1$ with its projection.  If this pair is isotrivial over $S'$, then there is a change of coordinates (passing to a further finite branched cover of $S'$ if necessary) so that the pair is constant.  In this case, we select a different periodic point $P_1$ for $f$.  If all periodic points of large period for $f$ lead to isotrivial pairs for $g$, then we carry out an interpolation argument as in \S\ref{dim 1} to deduce that $f$ must also be constant in the new coordinates over $S'$.  In other words, the pair $(f,g)$ is isotrivial, a contradiction.

So we may assume that there exists a periodic point $P_1$ for $f$ so that $P_1 \cap \Omega \not= \emptyset$ and the pair $(g, P_1)$ is not isotrivial over $S$.  It follows from Theorem \ref{stable point} that the pair $(g,P_1)$ is either persistently preperiodic, in which case we let $x_1$ be any element of $P_1\cap \Omega$, or there exists a point $x_1 = (s_1, z_1, \ldots, z_m) \in P_1 \cap \Omega$ so that $z_1$ is preperiodic to a repelling cycle of $g_{s_1}$.  We then consider an irreducible subvariety $P_1' \subset P_1$ containing $x_1$ along which $z_1$ is persistently preperiodic for $g$.  Note that the codimension of $P_1'$ is at most 1 and its projection to $S$ will have codimension at most 1.  If the codimension is 1, we let $S_1$ be its projection to the base.  If this codimension is 0, we replace $P_1'$ with its intersection with $\pi^{-1}(S_1)$ for an arbitrarily-chosen irreducible subvariety $S_1$ of codimension 1 in $S$ passing through $s_1$.  Therefore, $P_1' $ has nonempty intersection with $\Omega$, the projection of $P_1'$ to the base $S$ has dimension $m-1$, and the $z_1$-coordinate of $P_1'$ is persistently preperiodic for both $f$ and $g$.  If $S_1$ is singular, we replace it with the regular part, so that it will be a smooth and irreducible quasiprojective variety of dimension $m-1$.

We now repeat the process, beginning with a subvariety $P_2$ of codimension 1 in $P_1'$, having nonempty intersection with $\Omega$, and for which the second coordinate $z_2$ is periodic for $f$ over all of $S_1$, with $z_2$ not identically equal to $z_1$ throughout $P_2$.  Projecting to the $z_2$-coordinate, we consider the associated pair $(g,P_2)$, passing to a further finite branched cover $S' \to S$ if needed.  If this pair is isotrivial, we replace $P_2$ with another choice of periodic point for $f$; as above, if all periodic points of sufficiently large period for $f$ lead to isotrivial pairs $(g,P_2)$, then the pair $(f,g)$ would be isotrivial along $S_1$.  Here we use the assumption that $(f,g)$ is maximally non-isotrivial, not just non-isotrivial.  So we can find a $P_2$ that has nonempty intersection with $\Omega$ and so that -- when projecting to the $z_2$-coordinate -- the pair $(f,P_2)$ is persistently periodic and the pair $(g,P_2)$ is not isotrivial. 

In this way, we inductively reduce the problem to the argument of \S\ref{dim 1}.  This completes the proof of Theorem \ref{Zariski dense}.


\bigskip
\section{Nondegeneracy and the bifurcation measure}\label{nondegeneracy:bifmeasure}
\label{measure}

In this section, we work in the more general setting of families of polarized dynamical systems over a complex quasiprojective variety $S$.  We review some important notions introduced in \cite{Gauthier:Vigny} and \cite{Yuan:Zhang:quasiprojective} and remind the reader of their relations to the non-degeneracy conditions introduced by Habegger \cite{Habegger:special} and studied by Gao \cite{Gao:generic} 
for subvarieties in families of abelian varieties. Finally, in \S\ref{repeller}, we establish a criterion for non-degeneracy for certain families of polarized dynamical systems and subvarieties.

\subsection{The bifurcation current and measure for families of endomorphisms}  \label{gen bif current}
Suppose that $S$ is a smooth and irreducible quasiprojective variety defined over $\C$.  A {\bf family of ($k$-dimensional) polarized dynamical systems} $(X\to S, \Phi, \mathcal{L})$ is given by a family of complex projective varieties $X\to S$, flat over $S$ with smooth fibers $X_s$  of dimension $k$ over each $s\in S$, a regular map $\Phi: X\to X$ that preserves the fibers $X_s$, and a relatively ample line bundle $\mathcal{L}$ on $X$ such that for each $s\in S$, we have $(\Phi|_{X_s})^{*}(\mathcal{L}|_{X_s})\iso (\mathcal{L}|_{X_s})^{\otimes d}$ for some $d>0$.

\begin{example}  \label{pair example}
 Let $\Phi = (f,g)$ be an algebraic family of pairs over $S$, of degree $d\geq 1$, as considered in the previous section.  Then $\Phi$ defines a family of 2-dimensional polarized dynamical systems. The degree $d$ is the degree of a polarization of $\Phi$, taking line bundle $L = p_1^*O(1)\otimes p_2^*O(1)$ on $S \times \P^1\times\P^1$, where $p_i: S\times \P^1\times\P^1 \to \P^1$ is the projection to the $i$-th factor of $\P^1$, $i = 1,2$.   
\end{example}

As explained for instance in \cite[\S2.3]{Gauthier:Vigny}, to such a family we can associate a \textbf{dynamical Green current}, denoted by $\hat{T}_{\Phi}$, as follows.  We let $\hat{\omega}$ be a smooth positive $(1,1)$-form on $X$ cohomologous to a multiple of $\mathcal{L}$ such that $\omega_s :=\hat{\omega}|_{X_{s}}$ is K\"ahler for all $s\in S$ and
\begin{align}\label{normalize omega}
\int_{X_{s}}\omega^k_{s}=1
\end{align}
for each $s\in S$, where $k = \dim X_s$.  The sequence $d^{-n}(\Phi^n)^{*}(\hat{\omega})$ converges weakly to a closed positive $(1,1)$-current $\hat{T}_{\Phi}$ with continuous potentials.

\begin{example}  \label{normalized pair}
For $X = S \times\P^1$ and family of maps $f$, the dynamical Green current of $f$ coincides with the current defined in \eqref{Green}, taking $\hat\omega$ to be $p^*\omega$ for the projection $p: S\times\P^1\to \P^1$.  If $\Phi = (f,g)$ is an algebraic family of pairs over $S$, polarized as in Example \ref{pair example}, then $\hat{T}_\Phi = \frac{1}{\sqrt{2}} \left(p_1^*\hat{T}_f + p_2^*\hat{T}_g\right)$, taking $\hat{\omega} = \frac{1}{\sqrt{2}} \left( p_1^* \hat\omega + p_2^*\hat\omega\right)$ with $p_i: S\times \P^1\times\P^1 \to S \times \P^1$ the projection to the product of $S$ with the $i$-th factor of $\P^1$, $i = 1,2$.  Note that the constant $1/\sqrt{2}$ comes from the normalization \eqref{normalize omega}.
\end{example}

Suppose $(X \to S, \Phi, \mathcal{L})$ is a family of $k$-dimensional polarized dynamical systems.  Suppose that $Y$ is closed subvariety of $X$ of codimension equal to $r$, defining a flat family over $S$. As defined in \cite{Gauthier:Vigny}, the {\bf bifurcation current} for the triple $(X, Y, \Phi)$ is defined by 
\begin{equation} \label{bif current 2}
	\hat{T}_{\Phi, Y} := \pi_* \left(\hat{T}_{\Phi}^{\wedge (k-r+1)} \wedge [Y]\right)
\end{equation} 
where $\pi:X \to S$ is the projection.  The {\bf bifurcation measure} is given by 
\begin{equation} \label{bif measure 2}
	\mu_{\Phi, Y} := \left(\hat{T}_{\Phi, Y}\right)^{\wedge (\dim S)}
\end{equation}
on $S$. The wedge powers are well defined because the current has continuous potentials.  Also as a consequence of having a continuous potential, the bifurcation measure $\mu_{\Phi, Y}$ does not charge pluripolar sets in $S$ \cite[Proposition 4.6.4]{Klimek}.

\begin{example}\label{bb:bifmeasure}  Suppose that $f$ is an algebraic family of maps on $\P^1$ of degree $d > 1$, parameterized by a smooth and irreducible $S$.  Let $Y$ be the critical locus of $f$ in $S \times \P^1$.  Then the bifurcation current $\hat{T}_{f, Y}$ coincides with the bifurcation current $\hat{T}_{f, \mathrm{bif}}$ defined above in \eqref{T bif}, introduced in \cite{D:current, Dujardin:Favre:critical}, and the bifurcation measure coincides with that of \cite{Bassanelli:Berteloot}.
\end{example}

\begin{example}  \label{pairwise bif current}
For any algebraic family of pairs $\Phi = (f,g)$ over a smooth and irreducible complex quasiprojective variety $S$, with projection $\pi:  S\times (\P^1)^2 \to S$, we let $X = S \times (\P^1\times\P^1)$ and set $Y  = S \times \Delta$, where $\Delta \subset \P^1 \times\P^1$ is the diagonal.  Then we refer to $\hat{T}_{\Phi, Y}$ as the  \textbf{pairwise-bifurcation current associated to $\Phi$} and denote it by $\hat{T}_{\Phi,\Delta}$.  In the notation of Example \ref{normalized pair}, we have 
\begin{eqnarray*} 
\hat{T}_{\Phi, \Delta} &=& \pi_* \left(\hat{T}_{\Phi}^{\wedge 2} \wedge [S \times \Delta]\right) \\
	&=& \pi_* \left( p_1^*\hat{T}_f \wedge p_2^* \hat{T}_g \wedge [S \times \Delta] \right) \\
	&=& \pi_*' \left( \hat{T}_f \wedge \hat{T}_g \right),
\end{eqnarray*}
where $\pi'$ denotes the projection from $S \times \P^1$ to $S$.  In particular, when $S$ is the total space $\mathrm{Rat}_d\times \mathrm{Rat}_d$, this agrees with the pairwise-bifurcation current $\hat{T}_{\Delta}$ defined in \eqref{bif current 1}.
 The {\bf pairwise-bifurcation measure} is defined as 
	$$\mu_{\Phi, \Delta} = \left(\hat{T}_{\Phi, \Delta}\right)^{\wedge (\dim S)}.$$
\end{example}

\subsection{Non-degeneracy and equidistribution}
Suppose $(X \to S, \Phi, \cL)$ is a family of $k$-dimensional polarized dynamical systems.  Suppose that $Y$ is a closed subvariety of $X$ of codimension equal to $r$, defining a flat family of subvarieties in $X$.  We say that the triple $(X, Y, \Phi)$ is {\bf non-degenerate} if the current 
	$$\hat{T}_{\Phi}^{\wedge (k-r + \dim S)} \wedge [Y]$$ 
is non-zero on $X$. This is an exact analog of the notion of non-degeneracy introduced by Habegger in \cite{Habegger:special} and studied in general by Gao \cite{Gao:generic} for subvarieties in families of abelian varieties, where the Betti form is replaced by $\hat{T}_{\Phi}$ on the total space $X$. This notion of non-degeneracy agrees with the one introduced by Yuan and Zhang \cite[\S 6.2.2]{Yuan:Zhang:quasiprojective} as they demonstrate in \cite[Lemma 5.4.4]{Yuan:Zhang:quasiprojective}.

For a family of hypersurfaces $Y$, Gauthier, Taflin, and Vigny recently observed a relation between the bifurcation measure and non-degeneracy on a fiber power of $X$ \cite{Gauthier:Vigny:new}; compare \cite[Lemma 4.1]{Yuan:uniform}.  For any integer $m\geq 1$, let 
	$$\Phi^{(m)}: X^{(m)} \to X^{(m)}$$ 
be the fiber power of $\Phi$ acting on the $m$-th fiber power of $X$ over $S$.  It is polarized by the line bundle $p_1^* \cL\otimes \cdots \otimes p_m^* \cL$, where $p_i: X^{(m)} \to X$ is the projection to the $i$-th factor.  Let $Y^{(m)}$ be the corresponding fiber power of $Y$ over $S$.  We continue to denote the projection to $S$ by $\pi$.  

\begin{prop} \cite{Gauthier:Vigny:new}  \label{GV measure}
Suppose $(X \to S, \Phi, \cL)$ is a family of $k$-dimensional polarized dynamical systems, and let $m = \dim S$.  Assume that $Y$ is a closed hypersurface in $X$, defining a flat family of hypersurfaces over $S$.  Then the bifurcation measure $\mu_{\Phi, Y}$ on $S$ satisfies
	$$\mu_{\Phi, Y} \; = \; \left(\hat{T}_{\Phi, Y}\right)^{\wedge m} \; =  \; \pi_* \left( \hat{T}_{\Phi^{(m)}}^{\wedge (mk)} \wedge [Y^{(m)}] \right).$$
In particular, the triple $(X^{(m)}, Y^{(m)}, \Phi^{(m)})$ is non-degenerate if and only if the bifurcation measure $\mu_{\Phi, Y}$ is nonzero.  
\end{prop}

\begin{proof}
The first equality is simply the definition of $\mu_{\Phi,Y}$, and we need to prove the second.  Since the dimension of each fiber of $X \to S$ is $k$, it follows that $\hat{T}_\Phi^{\wedge (k+1)} = 0$ on $X$.  Let $m = \dim S$, and note that the fibers of $X^{(m)}\to S$ have dimension $mk$ and $Y^{(m)}$ has codimension $m$.  Let $p_j: X^{(m)} \to X$ be the projection to the $j$-th factor.  Then there is a constant $C>0$ so that
\begin{eqnarray*}
\hat{T}_{\Phi^{(m)}}^{\wedge (mk)} &=&  C \; \left(p_1^* \hat{T}_\Phi + \cdots + p_m^* \hat{T}_\Phi\right)^{\wedge (mk)} \\
	&=& C \; \frac{(mk)!}{(k!)^m} \bigwedge_{j=1}^m \, p_j^* \hat{T}_\Phi^{\wedge k}.
\end{eqnarray*}
Because of the normalization that $\hat{T}_{\Phi^{(m)}}^{\wedge (mk)}$ be a probability measure on each slice over $t \in S$, we must take $C = \frac{(k!)^m}{(mk)!}$.  On the other hand, we have 
	$$[Y^{(m)}] \; = \; \bigwedge_{j=1}^m \, p_j^*[Y],$$
so that
$$\hat{T}_{\Phi^{(m)}}^{\wedge (mk)} \wedge [Y^{(m)}] \; = \; \bigwedge_{j=1}^m  \, p_j^* \left(\hat{T}_\Phi^{\wedge k} \wedge [Y] \right)$$
and the conclusion follows.
\end{proof}

Now assume that the triple $(X, Y, \Phi)$ is defined over a number field and is non-degenerate.  Yuan and Zhang recently proved an equidistribution theorem in \cite{Yuan:Zhang:quasiprojective} for points of small fiber-wise canonical height in $Y$, extending a result of K\"uhne \cite{Kuhne:UMM}.  A closely related result has also recently been obtained by Gauthier \cite{Gauthier:goodheights}.  
A sequence of points $y_n \in Y(\Qbar)$ is said to be {\bf generic} if no subsequence lies in a proper, Zariski-closed subset of $Y$. 

\begin{theorem} \cite[Theorem 6.2.3]{Yuan:Zhang:quasiprojective} \label{equidistribution}  Suppose $(X \to S, \Phi, \cL)$ is a family of $k$-dimensional polarized dynamical systems over smooth, irreducible, quasiprojective $S$, defined over a number field $K$.  Suppose that the triple $(X, Y, \Phi)$ is non-degenerate, where $Y$ is a closed subvariety of $X$ of codimension $r$, defining a flat family over $S$, and also defined over $K$. Then for any generic sequence of preperiodic points of $\Phi$ in $Y(\Kbar)$, their $\Gal(\Kbar/K)$-orbits are uniformly distributed with respect to the measure $\hat{T}_{\Phi}^{\wedge (k-r+\dim S)} \wedge [Y]$ on $X(\C)$.  More precisely, given any continuous function $\phi$ with compact support in $X$, and given any generic sequence $\{y_n\}$ of points in $Y(\Kbar)$ that are preperiodic for $\Phi$, we have 
	$$\frac{1}{\# \Gal(\Kbar/K)\cdot y_n} \sum_{y \in \Gal(\Kbar/K)\cdot y_n} \phi(y) \; \longrightarrow \; \frac{1}{\vol(Y)} \int_{X(\C)} \phi \; \left(\hat{T}_{\Phi}^{\wedge (k-r+\dim S)} \wedge [Y]\right),$$
as $n\to\infty$ where $\vol(Y) = \int_{X(\C)} \hat{T}_{\Phi}^{\wedge (k-r+\dim S)} \wedge [Y]$.
\end{theorem}

\subsection{A repelling-cycle criterion to show non-degeneracy}  \label{repeller}
Suppose $(\pi: X \to S, \Phi, \cL)$ is a family of $k$-dimensional polarized dynamical systems.  Now suppose that $Y$ is a closed subvariety of $X$ of codimension $m = \dim S$ in $X$, defining a flat family of subvarieties with codimension $m$.  We say that a point $y_0 \in Y(\C)$ is a {\bf rigid repeller} for $(X, Y, \Phi)$ if 
\begin{enumerate}
\item  some iterate $x_0 = \Phi^{n_0}(y_0)$ is a repelling periodic point for $\Phi_{s_0}$, where $s_0 = \pi(y_0)$;  
\item  the point $y_0$ lies in the support of the canonical measure $\mu_{s_0} := \left(\hat{T}_\Phi|_{X_{s_0}}\right)^k$ in the fiber $X_{s_0}$; and
\item there is a holomorphic section $\eta$ over a neighborhood of $s_0$ in $S$ parameterizing a repelling periodic point of $\Phi_s$, with $\eta(s_0) = x_0$, so that $x_0$ is an isolated point of the intersection of $\Phi^{n_0}(Y)$ with the image of $\eta$.
\end{enumerate}

\begin{remark} 
If $\Phi = (f,g)$ is an algebraic family of pairs over $S$, as defined in Section \ref{density} and Example \ref{pair example}, with $m = \dim S$, a rigid repeller for the fiber product $\Phi^{(m)}: X^{(m)} \to X^{(m)}$ in $Y  = S \times \Delta^m \subset X^{(m)}$, where $X=S\times\P^1\times\P^1$, coincides with the notion of the rigid $m$-repeller for $\Phi$ from the Introduction.
\end{remark}

The next proposition is a minor modification of \cite[Lemma 4.8]{Gauthier:Vigny}, and the proof is very similar to that of \cite[Proposition 3.7]{Berteloot:Bianchi:Dupont}.  

\begin{prop} \label{transverse}
Suppose $(X \to S, \Phi, \cL)$ is a family of $k$-dimensional polarized dynamical systems over quasiprojective $S$.  Suppose that $Y$ is a closed subvariety of $X$ of codimension equal to $\dim S$, defining a flat family over $S$.  Suppose there exists a rigid repeller for $(X, Y, \Phi)$ at $y_0 \in Y(\C)$.  Then $y_0$ lies in the support of the (nonzero) measure
	$$\left(\hat{T}_{\Phi}\right)^{\wedge k} \wedge [Y]$$
on $X(\C)$.  In particular, the triple $(X, Y, \Phi)$ is non-degenerate.
\end{prop}

\begin{proof}  
Let $d \geq 2$ be the polarization degree of $\Phi$.  Let $x_0 = \Phi^{n_0}(y_0) \in \Phi^{n_0}(Y)$ be a repelling periodic point in the orbit of $y_0$.    Let $s_0 = \pi(x_0) \in S$.  Since $y_0$ is in the support of the canonical measure $\mu_{s_0} = \left(\hat{T}_\Phi|_{X_{s_0}}\right)^k$, it follows that $x_0$ is also in this support.   Let $p$ be the period of $x_0$.  Let $\eta$ denote the parameterization of the nearby repelling periodic points over a neighborhood $U$ of $s_0$, and let $\Gamma_\eta$ denote its image in $X$.  By hypothesis, $x_0$ is an isolated point of the intersection of $\Gamma_\eta$ with $\Phi^{n}(Y_0)$.

Shrinking $U$ if necessary, there exists a tubular neighborhood $N$ of $\Gamma_\eta$ in $\pi^{-1}(U)$ and a constant $K>1$ so that 
	$$d_{X_s}(\Phi_s^p(x), \Phi_s^p(\eta(s))) \geq K \, d_{X_s}(x, \eta(s))$$
for all $x \in N \cap X_s$ and all $s \in U$ and for any reasonable choice of distance function $d_{X_s}$ on the fibers.   In particular, there exists a nested sequence of tubular neighborhoods $N_n \subset N$ around $\Gamma_\eta \cap \pi^{-1}(U)$, for $n\geq 1$, so that $\Phi^{np}: N_n \to N$ is proper and one-to-one. Then for all integers $n\geq 0$, 
\begin{eqnarray*} 
d^{npk} \int_{N_n}  \hat{T}_\Phi^{\wedge k} \wedge [\Phi^{n_0}(Y)]   
	&=& \int_{N_n} (\Phi^{np})^* \hat{T}_\Phi^{\wedge k} \wedge [\Phi^{n_0}(Y)]  \\
	&=& \int_N \hat{T}_\Phi^{\wedge k} \wedge (\Phi^{np})_* [\Phi^{n_0}(Y)] \\
	&=& \int_N \hat{T}_\Phi^{\wedge k} \wedge [\Phi^{n_0+np}(Y)].
\end{eqnarray*}
On the other hand, we have 
	$$\lim_{n\to\infty} \; \chi_N \; [\Phi^{n_0+np}(Y)] \; =  \; \alpha \,  [X_{s_0}\cap N]$$
for some $\alpha>0$, in the weak sense of currents, where $\chi_N$ is the indicator function.  Indeed, since $N\cap \Phi^{n_0}(Y) \cap \Gamma_\eta = \{x_0\}$, the vertical expansion of $\Phi^p$ shows that the limit must be supported in the fiber $X_{s_0}$.  As the limit current is closed and positive and nonzero, it must be (a scalar multiple of) the current of integration along $X_{s_0}\cap N$.  Finally, since $x_0$ is in the support of the measure $\mu_{s_0} = \hat{T}_\Phi^{\wedge k}|_{X_{s_0}}$, we have
	$$\int_N  \hat{T}_\Phi^{\wedge k} \wedge [X_{s_0}\cap N]  \; > \; 0$$
so that 
	$$\int_N \hat{T}_\Phi^{\wedge k} \wedge [\Phi^{n_0+np}(Y)] \; > \; 0$$
for all sufficiently large $n$.

Now choose an open set $V$ around $y_0$ so that $\Phi^{n_0}: V \to N$ is proper.  Then 
\begin{eqnarray*} 
d^{n_0k} \int_V  \hat{T}_\Phi^{\wedge k} \wedge [Y]   
	&=& \int_V (\Phi^n)^* \hat{T}_\Phi^{\wedge k} \wedge [Y]  \\
	&=& \int_{N} \hat{T}_\Phi^{\wedge k} \wedge (\Phi^{n_0})_* [Y] \\
	&\geq& \int_{N} \hat{T}_\Phi^{\wedge k} \wedge [\Phi^{n_0}(Y)] \\
	&\geq& \int_{N_n}  \hat{T}_\Phi^{\wedge k} \wedge [\Phi^{n_0}(Y)]
\end{eqnarray*}
for all $n\geq 1$.  Therefore
	$$\int_V  \hat{T}_\Phi^{\wedge k} \wedge [Y]   \; > \; 0.$$
\end{proof}

Now let $S$ be a smooth and irreducible complex quasiprojective variety.  Recall that an algebraic family of pairs $\Phi = (f,g)$ was defined in Section \ref{density}, and the pairwise-bifurcation current was defined in Example \ref{pairwise bif current}.

\begin{cor} \label{our measure}
Let $S$ be a smooth and irreducible complex quasiprojective variety of dimension $m$.  Suppose that $\Phi = (f,g)$ is an algebraic family of pairs over $S$.  If there exists a rigid $m$-repeller at some parameter $s_0 \in S(\C)$, then the pairwise-bifurcation measure $\mu_{\Phi, \Delta}$ is nonzero on $S(\C)$.
\end{cor}

\begin{proof}  
As observed above, a rigid $m$-repeller for $\Phi$ implies there is a rigid repeller for the $m$-th fiber power $\Phi^{(m)}$ in $S \times \Delta^m \subset S \times (\P^1\times\P^1)^m$.  Proposition \ref{transverse} implies that $\left(\hat{T}_{\Phi^{(m)}}\right)^{\wedge 2m} \wedge [S \times \Delta^m]$ is nonzero.  Propsition \ref{GV measure} then implies that the measure $\mu_{\Phi, \Delta}$ is nonzero on $S$.  
\end{proof}

\bigskip
\section{Proof of Theorem \ref{non-density result}}\label{unlikely intersection}
We will deduce Theorem \ref{non-density result} from the following result, combined with the material of the previous section.  Recall that the pairwise-bifurcation measure $\mu_{\Phi, \Delta}$ was defined in Example \ref{pairwise bif current}.

\begin{theorem} \label{rbc}
Fix degree $d \geq 2$, and suppose that $S$ is a smooth, irreducible quasi-projective variety of dimension $m$ parameterizing an algebraic family of pairs $\Phi = (f,g)$ of degree $d\ge 2$ over $S$, all defined over $\Qbar$.  Assume that  $f$ and $g$ are not both conjugate to $z^{\pm d}$ over the algebraic closure of $k = \Qbar(S)$.  If the pairwise-bifurcation measure $\mu_{\Phi, \Delta}$ is non-zero on $S(\C)$, then the set of points 
$$\{(s, x_1, \ldots, x_{m+1}) \in (S \times (\P^1)^{m+1})(\Qbar):  x_i \in \Preper(f_s) \cap \Preper(g_s) \mbox{ for each } i\}$$
is {\em not} Zariski dense in $S \times (\P^1)^{m+1}$.  
\end{theorem}

\noindent
To prove Theorem \ref{rbc}, we follow the proof strategy from \cite{Mavraki:Schmidt}, which exploits the product structure of $(f,g)$ acting on $\P^1\times\P^1$ and relies on the general equidistribution result of Yuan and Zhang \cite{Yuan:Zhang:quasiprojective}, stated above as Theorem \ref{equidistribution}.  As a consequence, we infer:

\begin{theorem} \label{uniform bound}
Fix degree $d \geq 2$, and suppose that $S$ is a smooth, irreducible quasi-projective variety parameterizing an algebraic family of pairs $\Phi = (f,g)$ of degree $d\ge 2$ over $S$, all defined over $\Qbar$. 
Assume that  $f$ and $g$ are not both conjugate to $z^{\pm d}$ over the algebraic closure of $k = \Qbar(S)$. If the pairwise-bifurcation measure $\mu_{\Phi, \Delta}$ is non-zero on $S(\C)$, then there exist a Zariski-closed proper subvariety $V\subset S$ defined over $\Qbar$ and $M >0$ such that 
	$$\#\Preper(f_s)\cap\Preper(g_s) \; \le \; M,$$
for all $s\in (S\setminus V)(\C)$. 
\end{theorem}

\subsection{Product structure}
Let $\Phi = (f,g)$ act on $\P^1\times\P^1$, defined over the field $k = \Qbar(S)$.  Let $\Delta \subset \P^1\times\P^1$ be the diagonal.  For $m = \dim S$, we let $\Phi^{(m)}$ denote the product map acting on $(\P^1 \times \P^1)^m$ over $k$.  Following \cite{Mavraki:Schmidt}, we consider the product of $(\P^1\times\P^1)^{m}$ with another copy of $\P^1$, acted on by $f$ or by $g$.  This defines maps over $\C$ as
	$$(\Phi^{(m)},f):  S \times (\P^1)^{2m + 1} \to S \times (\P^1)^{2m + 1}$$
	$$(s, z_1, \ldots, z_{2m + 1}) \mapsto \big(s, f_s(z_1), g_s(z_2), f_s(z_3),  \ldots, g_s(z_{2 m}), f_s(z_{2m + 1})\big)$$
and 
	$$(\Phi^{(m)},g):  S \times (\P^1)^{2m + 1} \to S \times (\P^1)^{2m + 1}$$
	$$(s, z_1, \ldots, z_{2m + 1}) \mapsto \big(s, f_s(z_1), g_s(z_2), f_s(z_3), \ldots, g_s(z_{2 \dim S}), g_s(z_{2m + 1})\big).$$
Let 
	$$p_1: S \times (\P^1\times\P^1)^{m}\times \P^1 \to S \times (\P^1\times\P^1)^{m}$$
be the projection forgetting the final factor of $\P^1$.  Let 
	$$p_2:  S \times (\P^1\times\P^1)^{m}\times \P^1 \to  S \times  \P^1$$
denote the projection forgetting the intermediate factor.

\begin{prop}  \label{higher volumes}
Let $S$ be an irreducible quasiprojective complex algebraic variety of dimension $m$, and suppose $\Phi = (f,g)$ is an algebraic family of pairs of degree $d>1$ over $S$.  We have
$$M_f := \hat{T}_{(\Phi^{(m)},f)}^{\wedge(2m + 1)} \wedge [S \times \Delta^{m}\times \P^1] = 
p_1^*R \wedge p_2^*\hat{T}_f,$$
where 
	$$R =  \hat{T}_{\Phi^{(m)}}^{\wedge 2m} \wedge [S \times \Delta^{m}].$$
Similarly for $g$ and 
	$$M_g := \hat{T}_{(\Phi^{(m)},g)}^{\wedge(2m + 1)} \wedge [S \times \Delta^{m}\times \P^1].$$
Consequently, if the bifurcation measure $\mu_{\Phi, \Delta}$ is nonzero for a family of pairs $\Phi = (f,g)$ parameterized by $S$, then $M_f$ and $M_g$ are nonzero.  
\end{prop}

\begin{proof}  
Suppose that $f$ is an algebraic family of maps on $\P^1$ over $S$.  Then $\hat{T}_f$ on $S \times \P^1$ satisfies $\hat{T}_f^{\wedge 2} = 0$.  It follows that, for any fiber product of such maps, $(f_1, \ldots, f_\ell)$ on $(\P^1)^\ell$ over $S$, we have 
	$$\hat{T}_{(f_1, \ldots, f_\ell)}^{\wedge \ell}= q_1^* \hat{T}_{f_1} \wedge \cdots \wedge q_m^* \hat{T}_{f_\ell}$$
for the projections $q_i: S\times (\P^1)^\ell \to S\times \P^1$.  In the setting of the proposition, it follows that 
	$$\hat{T}_{\Phi^{(m)}}^{\wedge(2m)} = q_1^* \hat{T}_f \wedge q_2^* \hat{T}_g \wedge \cdots \wedge q_{2m - 1}^* \hat{T}_f \wedge q_{2m}^* \hat{T}_g$$
and
	$$\hat{T}_{(\Phi^{(m)},f)}^{\wedge(2m + 1)} = p_1^* \hat{T}_{\Phi^{(m)}}^{\wedge(2m)}  \wedge p_2^* \hat{T}_f.$$
The first statements of the proposition follow. Finally, since $\Delta \subset \P^1\times\P^1$ is a hypersurface, we know from Proposition \ref{GV measure} that $\mu_{\Phi, \Delta}$ is nonzero if and only if 
	$$R = \hat{T}_{\Phi^{(m)}}^{\wedge(2m)} \wedge [S \times \Delta^m] > 0.$$
In this case, we see immediately that $M_f$ and $M_g$ are also nonzero.
\end{proof}

\subsection{Proof of Theorem \ref{rbc}}  
Recall that a sequence of points $z_n$ in a variety $Z$ is said to be generic if no subsequence lies in a proper, Zariski-closed subset of $Z$.

Let $m=\dim_\C S$.  Note that the set 
 $$\{(s, x_1, \ldots, x_{m+1}) \in (S \times (\P^1)^{m+1})(\Qbar):  x_i \in \Preper(f_s) \cap \Preper(g_s) \mbox{ for each } i\}$$
in the statement of Theorem \ref{rbc} is naturally identified with the set 
	$$\Preper(\Phi^{(m+1)}) \cap \left(S(\Qbar) \times \Delta^{m+1}\right) \; \subset \; S\times (\P^1\times\P^1)^{m+1},$$
for the $(m+1)$-th fiber power of $\Phi$ over $S$. 
 
Let $K$ be a number field over which $\Phi$ and $S$ are defined.  Suppose, towards a contradiction, that $\Preper(\Phi^{(m+1)})$ is Zariski dense in $S(\Kbar)\times \Delta^{m+1}$.  Via the projection of $\Delta \subset \P^1\times\P^1$ to the component $\P^1$'s, these preperiodic points of $\Phi^{(m+1)}$ in $S\times \Delta^{m+1}$ project to define a generic sequence of points in $\left(S\times\Delta^{m}\times \P^1\right)(\Kbar)$ that are preperiodic for both the maps $(\Phi^{(\dim S)}, f)$ and $(\Phi^{(\dim S)}, g)$.  Since the pairwise-bifurcation measure $\mu_{\Phi, \Delta}$ is nonzero on $S(\C)$, we know from Proposition \ref{higher volumes} that the measures $M_g$ and $M_f$ are nonzero on $\left( S\times(\P^1\times\P^1)^m\times \P^1 \right)(\C)$.  In other words, setting 
	$$X = S\times(\P^1\times\P^1)^m\times \P^1 \quad\mbox{and}\quad Y = S\times\Delta^{m}\times \P^1,$$
the triples $(X, Y, (\Phi^{(m)},f))$ and $(X, Y,  (\Phi^{(m)},g))$ are non-degenerate.  By Theorem \ref{equidistribution}, it follows that the $\Gal(\Kbar/K)$-orbits of these preperiodic points must be uniformly distributed with respect to the measures $M_f$ and $M_g$.  Consequently, we have $M_f = M_g$ in $X(\C)$.

Now let $p_1: S \times (\P^1\times\P^1)^{m}\times \P^1 \to S \times (\P^1\times\P^1)^{m}$
be the projection forgetting the final factor of $\P^1$, as in Proposition \ref{higher volumes}.  By slicing $M_f$ and $M_g$, we conclude that
\begin{equation} \label{equality slice}
	\int_{S\times(\P^1)^{2m}} \int_{\P^1}\phi(t,x)  \, d\mu_{f_{\pi(t)}}(x) \, dR(t) = \int_{S\times(\P^1)^{2m}} \int_{\P^1}\phi(t,x) \, d\mu_{g_{\pi(t)}}(x) \, dR(t)
\end{equation}
for every continuous and compactly supported function $\phi$ on $S\times (\P^1)^{2m+1}$ and the $R$ of Proposition \ref{higher volumes}.  Here, $\pi: S\times (\P^1)^m \to S$ denotes the projection to the base, and $\mu_{f_{\pi(t)}}$ and $\mu_{g_{\pi(t)}}$ are the measures of maximal entropy introduced in \S\ref{mme}.  But since $\pi_*R = \mu_{\Phi, \Delta}$, we infer that 
	$$\mu_{f_s} = \mu_{g_s}$$ 
on $\P^1$ for $\mu_{\Phi, \Delta}$-almost every parameter $s \in S(\C)$. Indeed, suppose there exists $b \in S(\C)$ with $\mu_{f_b} \not= \mu_{g_b}$ and so that $\mu_{\Phi,\Delta}(U) >0$ for every open neighborhood $U$ of $b$.  Then we can find a continuous function $\psi$ on $\P^1(\C)$ such that $\int \psi \mu_{f_b} \neq \int \psi \mu_{g_b}$.  By continuity of the measures we find that $\int \psi \mu_{f_t} \neq \int \psi \mu_{g_t}$ for all $t$ in a neighborhood $U$ of $b$. Therefore, setting $\phi(t, x_1, \ldots, x_{2m}, x_{2m+1}) = h_b(t) \psi(x_{2m+1})$ on $S \times (\P^1)^{2m+1}$ for a bump function $h_b$ supported in $U$, the equality \eqref{equality slice} will fail.

Now we use the hypothesis that $f$ and $g$ are not both conjugate to a power map $z^{\pm d}$ over all of $S$.  Let $V \subset S$ be the (possibly empty) proper subvariety over which the maps $f$ and $g$ are both conjugate to $\pm z^d$.  It follows from Theorem \ref{LPYZ} that $\Preper(f_s) = \Preper(g_s)$ for all $s\in \left(\supp \mu_{\Phi, \Delta} \setminus V\right)(\C)$.  As $\mu_{\Phi,\Delta}$ does not charge pluripolar sets, we conclude that $\Preper(f_s) = \Preper(g_s)$ for all $s \in \left(S \setminus V\right)(\C)$.  Indeed, the preperiodic points of $f$ or $g$ each form a countable union of hypersurfaces in $(S\setminus V) \times\P^1$.  For each irreducible hypersurface $P \subset S\times\P^1$ which is preperiodic for $f$, its intersection with $\Preper(g)$ contains all of  $P \cap \pi^{-1}(\supp \mu_{\Phi, \Delta} \setminus V)$ and so cannot lie in a countable union of hypersurfaces of $P$.  Therefore $P$ must be persistently preperiodic for $g$.  This shows that $\Preper(f_s) \subset \Preper(g_s)$ for all $s \in (S\setminus V)(\C)$; equality follows from the same argument in reverse.

Again from Theorem \ref{LPYZ}, it follows that $\mu_{f_s} = \mu_{g_s}$ for all $s \in (S\setminus V)(\C)$.  Consequently, $\hat{T}_f = \hat{T}_g$ on $(S\setminus V) \times \P^1$.  But this implies that the pairwise-bifurcation currrent $\hat{T}_{\Phi, \Delta} = \pi_*(\hat{T}_f \wedge \hat{T}_g)$ vanishes on $S\setminus V$.  Recalling that $\hat{T}_{\Phi, \Delta}$ has continuous potentials, so its support cannot lie in $V$, we conclude that $\hat{T}_{\Phi, \Delta} = 0$ and therefore $\mu_{\Phi, \Delta} = 0$ on all of $S(\C)$.  This contradicts our hypothesis.
 \qed

\subsection{Proof of Theorem \ref{non-density result}}
We assume there exists a rigid $m$-repeller for the family $(f,g)$ over $S$, where $m = \dim S$. Corollary \ref{our measure} implies that the bifurcation measure $\mu_{\Phi, \Delta}$ is nonzero on the space $S$.   Finally, Theorem \ref{rbc} gives us the result we desire. 
\qed

\subsection{Proof of Theorem \ref{uniform bound}}
Recall that a sequence of points $s_n \in S$ is said to be generic if no subsequence lies in a proper, Zariski-closed subset of $S$.  
We start with the following lemma. 
\begin{lemma}\label{Bezout}
Let $\Phi = (f,g)$ be an algebraic family of pairs parameterized by a smooth and irreducible $S$, defined over $\C$.  If there is a generic sequence of points $s_n \in S(\C)$, $n\geq 1$, over which the number of common preperiodic points for $f_{s_n}$ and $g_{s_n}$ is either infinite or increasing to $\infty$, then the preperiodic points of $\Phi^{(m)}$ are Zariski dense in $S \times \Delta^m$ for every $m \geq 1$.
\end{lemma}

\begin{proof}
Fix $m \geq 1$. For each $n$, let 
	$$M(n) \leq \# \; \Preper(f_{s_n}) \cap \Preper(g_{s_n}) \; \in  \; \mathbb{N} \cup \{\infty\} $$
so be chosen so that $M(n) \to \infty$ as $n\to \infty$.  The points of $\Preper(f_{s_n}) \cap \Preper(g_{s_n})$ determine a configuration of $M(n)^m$ points in $\Delta^m$ that are preperiodic for $\Phi_{s_n}^{(m)}$.  Note that this collection of points is symmetric under permutation of the coordinates on $\Delta^m$.  

Let $Z$ be the Zariski closure of these points in $S \times \Delta^m$.   The symmetry of the preperiodic points in $\Delta^m$ implies that $Z$ is also symmetric under permuting the $m$ coordinates of $\Delta^m$.  Moreover, because the sequence $\{s_n\}$ is Zariski dense in $S$, we see that $\pi(Z) = S$ for the projection $\pi: S\times \Delta^m \to S$.  

Now suppose that $Z$ is not all of $S\times \Delta^m$.  Then, over a Zariski-open subset $U\subset S$, $Z$ is contained in a family of hypersurfaces in $\Delta^m \iso (\P^1)^m$ over $U$ that are symmetric symmetric under permuting the $m$ coordinates.  Shrinking $U$ if necessary, these hypersurfaces will have a well-defined degree $(r, \ldots, r)$ for some $r \geq 1$; that is, each projection that forgets one component $\Delta$ will be of degree $r$.  But since the sequence $\{s_n\}$ is generic, this implies that $M(n) \leq r$ for all sufficently large $n$.  This is a contradiction.  
\end{proof}

Now to prove Theorem \ref{uniform bound}, let $S$ be a smooth, irreducible quasi-projective variety parameterizing an algebraic family of pairs $(f,g)$, all defined over $\Qbar$ as in its statement. By Theorem \ref{rbc} and in view of Lemma \ref{Bezout} we infer that there exists a strict Zariski closed $V\subset S$ defined over $\Qbar$ and $M\in\R$ such that 
\begin{align}\label{bound}
\#\Preper(f_s)\cap\Preper(g_s)\le M,
\end{align}
for all $s\in (S\setminus V)(\Qbar)$. Write $U:=S\setminus V$. 
We want to show that $M$ can be chosen so that \eqref{bound} holds for all $s\in U(\C)$. 
Fix $s_0\in  U(\C)\setminus U(\Qbar)$ and let $P_{s_0}:= \Preper(f_{s_0})\cap\Preper(g_{s_0})$.  Let $L$ be a finitely generated subfield of $\C$ (with transcendence degree at least $1$) such that $\Phi_{s_0}$ is defined over $L$.  So there is a quasi-projective variety $X$ over $\Qbar$ of finite type with function field $L$ over which we can extend $\Phi_{s_0}$ (viewing $s_0$ as an element of $L$) to an endomorphism $\Phi_X: X\times \P^1\times\P^1\to X\times \P^1\times\P^1$ defined over $\Qbar$.  We also extend $P_{s_0}$ to $P_X\subset X\times \P^1\times\P^1$.  Note that each specialization $(\Phi_X)_t = (f_t, g_t)$ for $t\in X(\Qbar)$ is naturally identified with some $\Phi_s$ for $s \in U(\Qbar)$. Thus, for each $t\in X(\Qbar)$ we have a uniform bound on $\# \Preper(f_t)\cap\Preper(g_t)$.  But clearly the specializations of the distinct points in $P_X$ remain distinct at some $t\in X(\Qbar)$. The proof is complete. 
\qed

\bigskip
\section{Quadratic polynomials}
\label{polynomials}

Before proceeding to the proof of Theorem \ref{general bound}, we present in this section a new proof of Theorem \ref{quad}.  The strategy of proof is the same as for Theorem \ref{general bound}, but the argument for proving that the pairwise-bifurcation measure is non-zero is considerably simpler for these pairs of quadratic polynomials.   On the other hand, because the parameter space is two dimensional, we can use Theorem \ref{MS} to complete the proof.

\begin{theorem}  \label{quad examples}
Let $f_c(z) = z^2+c$ for $c\in \C$, and consider the algebraic family of pairs $\Phi_{(c_1, c_2)} = (f_{c_1}, f_{c_2})$ parameterized by $(c_1, c_2)\in \C^2$.  The pairwise-bifurcation measure $\mu_{\Phi, \Delta}$ is non-zero on $\C^2$. 
\end{theorem}

\begin{proof}
We appeal to Corollary \ref{our measure} and study the pair 
	$$f_{-21/16}(z) = z^2-21/16 \quad \mbox{ and } \quad  f_{-29/16}(z) = z^2-29/16$$
These two polynomials have at least 26 common preperiodic points in $\C$; see \cite{Doyle:Hyde} for a construction of these quadratic polynomials and similar examples in higher degrees.  We will show that two of the common preperiodic points define a rigid $2$-repeller over $S = \C^2$.

Let $(a_0, b_0) = (-21/16, -29/16)$.  Suppose $p_1, p_2\in \C$ are common preperiodic points, each iterating to a repelling cycle for both $f_{a_0}$ and $f_{b_0}$, and suppose that there is a holomorphic map 
	$$R(c_1, c_2) =  (p_1^a(c_1), p_1^b(c_2), p_2^a(c_1), p_2^b(c_2)) \in \C^4$$ 
for $(c_1, c_2)$ near $(a_0, b_0)$ in $\C^2$ so that $R(a_0, b_0) = (p_1, p_1, p_2, p_2) \in \Delta^2$ and $R(c_1, c_2)$ is persistently preperiodic for the fiber power $\Phi^{(2)}$ over $S$.  Then the pair $(p_1, p_2)$ will form a rigid 2-repeller in $(\P^1 \times\P^1)^2$ at $(a_0,b_0)$ if
	$$\det \left(  \begin{array}{cccccc} 
		1 & 0 & 0 & 0 & 1 & 0  \\
		0 & 1 & 0 & 0 & 0 & 1   \\
		0 & 0 & 1 & 0 & p_{1,a}' & 0 \\
		0 & 0 & 1 & 0 & 0 & p_{1,b}' \\
		0 & 0 & 0 & 1 & p_{2,a}' & 0 \\
		0 & 0 & 0 & 1 & 0 & p_{2,b}'  \end{array} 
		\right)  \not= 0 $$
for any such $R$, where $p_{i,a}'$ and $p_{i,b}'$ ($i = 1,2$) denote the derivatives of the coordinate functions of $R$, evaluated at $c_1 = a_0$ and $c_2=b_0$, respectively.  The first four columns are a basis for the tangent space to $S \times \Delta^2$, while the second two columns span the tangent space to the graph of $R$ over a neighborhood of $(a_0,b_0) \in S$.  In fact, showing this determinant is nonzero is stronger than being a rigid repeller, since this will show the graph of $R$ intersects $\Delta^2$ transversely over $(a_0, b_0)$.  

A simple computation shows that the above determinant is equal to
	$$\det \left( \begin{array}{cc}   p_{1,a}' & p_{1,b}' \\ p_{2,a}' & p_{2,b}' \end{array} \right). $$

Now let us take $p_1 = 5/4$ and $p_2 = -7/4$.   The orbits of $p_1$ and $p_2$ for $f_{a_0}$ are 
	$$\frac54 \mapsto \frac14 \mapsto -\frac54 \mapsto \frac14$$
	$$-\frac74 \mapsto \frac74 \mapsto \frac74$$
with each landing on a repelling cycle.  The orbits of $p_1$ and $p_2$ for $f_{b_0}$ are 
	$$\frac54 \mapsto -\frac14 \mapsto -\frac74 \mapsto \frac54$$
with each in a repelling cycle of period 3.  To compute $p_{1,a}', p_{1,b}', p_{2,a}', p_{2,b}'$ we determine the equations of these cycles, as a function of the parameter $c$, and use implicit differentiation.  

The equation for a (strictly) preperiodic point $z$ so that $f_c(z)$ is in a cycle of period 2 is
	$$P_1(c,z) = 1 + c - z + z^2 = 0,$$
so that 
	$$p_{1,a}' = -\frac{\del P_1}{\del  c}\bigg/\frac{\del P_1}{\del z} \bigg|_{c = a_0, z = p_1} = -2/3.$$
The prefixed points for $f_c$ satisfy
	$$P_2(c,z) = c + z + z^2 = 0,$$
so
	$$p_{2,a}' = -\frac{\del P_2}{\del  c}\bigg/ \frac{\del P_2}{\del z} \bigg|_{c = a_0, z = p_2}  =  2/5.$$
The period-three cycles for $f_c$ satisfy
\begin{eqnarray*} 
P_3(c,z) & =&   1 + c + 2 c^2 + c^3 + z + 2 c z + c^2 z + z^2 + 3 c z^2 + \\
 && 3 c^2 z^2 + z^3 + 2 c z^3 + z^4 + 3 c z^4 + z^5 + z^6.
\end{eqnarray*}
This gives 
	$$p_{1,b}' = -\frac{\del P_3}{\del  c}\bigg/\frac{\del P_3}{\del z} \bigg|_{c = b_0, z = p_1} =  2/9$$
	$$p_{2,b}' = -\frac{\del P_3}{\del  c}\bigg/\frac{\del P_3}{\del z} \bigg|_{c = b_0, z = p_2} =  2/9$$
We conclude that 
	$$\det \left( \begin{array}{cc}   p_{1,a}' & p_{1,b}' \\ p_{2,a}' & p_{2,b}'  \end{array} \right) = -\frac{32}{135} \not= 0. $$
With Corollary \ref{our measure}, this completes the proof of Theorem \ref{quad examples}.  
\end{proof}

\subsection{Proof of Theorem \ref{quad}}
Let $S = \C^2$.  In view of Theorem \ref{quad examples}, Theorem \ref{uniform bound} implies that there is a finite collection of irreducible, algebraic curves $C_1,\ldots, C_m\subset S$ all defined over $\overline{\Q}$ and a constant $B$ so that 
\begin{align}\label{outoffinitelymany}
|\mathrm{Preper}(f_{t_1})\cap \mathrm{Preper}(f_{t_2})|\le B,
\end{align}
for all $(t_1,t_2)\in \C^2 \setminus \left(\bigcup_i C_i\right)$. Applying Theorem \ref{MS} to each $C_i$ we see that (enlarging $B$ if necessary) the bound in \eqref{outoffinitelymany} holds for each $(t_1,t_2)\in\C^2$ unless $\mathrm{Preper}(f_{t_1})=\mathrm{Preper}(f_{t_2})$. The latter happens only if the Julia sets of $f_{t_1}$ and $f_{t_2}$ coincide, and so by \cite{Baker:Eremenko}, only if $t_1=t_2$, which completes our proof. 
\qed

\bigskip
\section{Monomials}
\label{monomials}

In this section, we complete the proof of Theorem \ref{general bound}.  To achieve this, we will prove:

\begin{theorem} \label{monomial theorem}
For each degree $d\geq 2$, the pair $(z^d, \zeta \, z^d)$ for primitive root of unity $\zeta^{d+1}=1$ has a rigid $(4d-1)$-repeller.  
\end{theorem}

Recall that the bifurcation measure $\mu_\Delta$ on the $(4d-1)$-dimensional moduli space of pairs $(\Rat_d \times \Rat_d)/\Aut \P^1$ was defined in \eqref{bif measure 1}.  Via Corollary \ref{our measure}, Theorem \ref{monomial theorem} will imply that $\mu_\Delta$ is nonzero. But we must be careful:  the moduli space is likely to be singular at pairs $(f,g)$ with automorphisms, and this pair $(z^d, \zeta \, z^d)$ has automorphisms of the form $A(z) = \omega z$ for $\omega^{d-1} = 1$.  Throughout this section, we work with the subspace 
	$$S_d \subset \Rat_d \times \Rat_d$$ 
consisting of pairs $(f,g)$ where 
	$$f(z) = \frac{z^d + a_{d-1} z + \cdots + a_1 z}{b_{d-1}z^{d-1} + \cdots + b_1 z + \left(1 + \sum_{i = 1}^{d-1} a_i -\sum_{j=1}^{d-1} b_j\right)} $$
with $a_i, b_j \in \C$ and $g$ is arbitrary.  Note that $S_d$ is a smooth and irreducible quasiprojective complex algebraic variety.  This normalization for $f$ fixes the three elements of $\{0,1,\infty\}$ in $\P^1$, and the projection from $S_d$ to the moduli space $(\Rat_d \times \Rat_d)/\Aut \P^1$ is finite-to-one.  In other words, this $S_d$ defines a maximally non-isotrivial algebraic family of pairs of degree $d$, with $\dim S_d = 4d-1$.  It is not surjective to the moduli space of pairs, but it covers a Zariski open subset.  The pair $(z^d, \zeta z^d)$ is an element of $S_d$ for any choice of primitive $(d+1)$-th root of unity $\zeta$.  

\subsection{Proof of Theorem \ref{general bound}, assuming Theorem \ref{monomial theorem}} 
Let $\mu_{\Phi, \Delta}$ denote the pairwise-bifurcation measure on $S_d(\C)$ for the family of all pairs $\Phi = (f,g)$ parameterized by $S_d$, as defined in \eqref{bif measure 2} and Example \ref{pairwise bif current}.  With Theorem \ref{monomial theorem}, we may apply Corollary \ref{our measure} to deduce that the pairwise-bifurcation measure $\mu_{\Phi, \Delta}$ is non-zero on $S_d(\C)$.  We then apply Theorem \ref{uniform bound} to conclude that there is a Zariski-open subset $U$ of $S_d$, defined over $\Qbar$, for which there is a uniform bound on the number of common preperiodic points of $f_s$ and $g_s$ for all $s \in U(\C)$.  Taking the union of all ($\Aut \P^1$)-orbits of $U$ in $\Rat_d \times \Rat_d$ completes the proof.
\qed

\subsection{Rigidity of the monomial pair}  \label{special rigidity}
We now aim to prove Theorem \ref{rigidity of special pair}.  Fix degree $d\geq 2$.  Let $f_0(z) = z^d$ and $g_0(z) = \zeta z^d$ for $\zeta = e^{2\pi i/(d+1)}$.  By conjugating the image of $\psi$ and shrinking the domain disk if necessary, we may assume that its image lies in the subvariety $S_d \subset \Rat_d\times\Rat_d$.  So, let $\D\subset \C$ denote the unit disk, and suppose that $\psi = (\psi_1, \psi_2): \D\to S_d$ is a holomorphic map with $\psi(0) = (f_0, g_0)$ so that 
	$$\Preper(\psi_1(t)) \cap J(\psi_1(t)) =  \Preper(\psi_2(t)) \cap J(\psi_2(t))$$
for all $t \in \D$.  We need to show that $\psi$ is constant.

\begin{remark}  \label{poly examples}
The conclusion of Theorem \ref{rigidity of special pair} is false if we allow $\zeta$ to be a root of unity of any order $\leq d$.   For each $m \leq d$ and $\zeta$ with $\zeta^m = 1$, let
	$$f_c(z) = z^{d-m}(z^m + c) \quad\mbox{and}\quad g_c(z) = \zeta f_c(z)$$
for $c\in\C$.  Then $\zeta$ is a symmetry of the Julia set of $f_c$, and $f_c(\zeta z) = \zeta^{d-m} f_c(z) = \zeta^d f_c(z)$.  Note that $g_c^n(z) = \zeta^{1 + \cdots + d^{n-1}} f^n_c(z)$ for all $n$ and all $c$.  If a point $x$ is preperiodic for $f_c$, then the iterates $g_c^n(x)$ must eventually cycle, and vice versa.  That is, $\Preper(f_c) = \Preper(g_c)$ for all $c\in \C$ and $J(f_c) = J(g_c)$ for all $c \in \C$.  See, for example, \cite{Baker:Eremenko} for more information on symmetries. 
\end{remark}

Returning to our setting, where $\zeta = e^{2\pi i/(d+1)}$, note that the second iterates of $f_0$ and $g_0$ coincide.  We first observe that the same relation must hold throughout $\D$; that is, we have 
	$$\psi_1(t)^2= \psi_2(t)^2$$
for all $t\in \D$.  Indeed, both $f_0$ and $g_0$ are $J$-stable in $\Rat_d$, and so there is a holomorphic motion of the Julia sets, inducing conjugacies between $\psi_i(0)$ and $\psi_i(t)$ on their Julia sets for $t$ small, $i = 1,2$. In particular, because all preperiodic points of $\psi_i(t)$, $i=1,2$, in the Julia set must coincide for all $t$, the motions $z_t$ of these preperiodic points $z \in J(\psi_1(0)) = J(\psi_2(0))$ must coincide for $\psi_1(t)$ and for $\psi_2(t)$, for all $t$ small.  The induced conjugacy forces a relation on the second iterates, $\psi_1(t)^2(z_t) =  \psi_2(t)^2(z_t)$, holding for all preperiodic points $z \in J(\psi_1(0)) = J(\psi_2(0))$ for all $t$ small.  As there are infinitely many preperiodic points in the Julia set, we have equality of iterates $\psi_1(t)^2 = \psi_2(t)^2$ for all $t$ small. Finally, by holomorphic continuation, the equality persists for all $t \in \D$.

Now consider the map $F = f^2-g^2$ from the space $\Rat_d\times\Rat_d$ to the set of all rational functions $R_{2d^2} \subset \C(z)$ of degree at most $2d^2$. Recall that the image of $\psi$ lies in the subspace $S_d$, which maps finite-to-one to $(\Rat_d\times\Rat_d)/\Aut\P^1$, and we have shown that $F(\psi(t)) = 0$ for all $t \in \D$.  We aim to conclude that $\psi$ is constant.  Choosing coefficients for coordinates on $\Rat_d\times\Rat_d$ near the point $(f_0, g_0)$ and on the target space $R_{2d^2}$, it is enough to show that the derivative matrix $DF_{(f_0, g_0)}$ has the maximal possible rank of $4d-1$.  Indeed, since $F$ vanishes on the fiber of the quotient $\Rat_d\times\Rat_d \to (\Rat_d\times\Rat_d)/\Aut\P^1$ through $(f_0, g_0)$, this will imply, by the Chain Rule, that some (possibly higher-order) derivative of $t \mapsto F(\psi(t))(z)$ will be nonzero at $t=0$, whenever $\psi$ is non-constant. 

We have thus reduced the proof of Theorem \ref{rigidity of special pair} to the following lemma:

\begin{lemma} \label{linear algebra}
For $d\ge 2$, let $f(z)=\frac{a_dz^d+\cdots+a_0}{b_dz^d+\cdots+b_1z+1}$ and $g(z)=\frac{A_dz^d+\cdots+A_0}{B_dz^d+\cdots+B_1z+1}$. 
Let $F=f^2-g^2$. Let $\vec{a}=(1,0,\ldots,0)$ corresponding to $z^d$ and $\vec{A}=(\zeta,0,\ldots,0)$ corresponding to $\zeta z^d$ for $\zeta^{d+1}=1$ a primitive $(d+1)$-th root of unity. 
Then the vectors $\partial_{a_i}F(\vec{a}), \partial_{b_j}F(\vec{a}), \partial_{A_k}F(\vec{A}), \partial_{B_{\ell}}F(\vec{A})$, for all $i,k\in\{0,\ldots,d\}$ and $j,\ell\in\{1,\ldots, d\}$ generate a subspace of $\C[z]$ of dimension $4d-1$. 
\end{lemma}

\begin{proof}
First we compute the derivatives 
\begin{align}
\begin{split}
\partial_{A_k}F(\vec{A})&=-\frac{1}{\zeta}d z^{d^2-d+k}-\zeta^{k}z^{dk}\\
\partial_{B_{\ell}}F(\vec{A})&=\zeta^{\ell}z^{d^2+\ell d}+dz^{d^2+\ell}\\
\partial_{a_i}F(\vec{a})&=dz^{d^2-d+i}+z^{id}\\
\partial_{b_j}F(\vec{a})&=-z^{d^2+jd}-dz^{d^2+j}. 
\end{split}
\end{align}
Let $M$ be the matrix with $s$-th column consisting of the coefficients of $z^{s-1}$ as they occur in order $\partial_{A_0}F(\vec{A}),\ldots,\partial_{A_d}F(\vec{A})$, $\partial_{B_1}F(\vec{A}), \ldots, \partial_{B_d}F(\vec{A})$, $\partial_{a_0}F(\vec{a}), \ldots, \partial_{a_d}F(\vec{a})$, $\partial_{b_1}F(\vec{a}),\ldots,\partial_{b_d}F(\vec{a})$. 
Each is a polynomial in $z$ with degree at most $2d^2$, so this a $(4d+2)\times (2d^2+1)$ matrix. 

Set $\{e_i\}$ to be the standard basis vectors for $\C^{4d+2}$. Notice that all powers of $z$ that appear, appear twice with the exception of $z^{d(d-1)}=z^{d^2-d}$ and $z^{d^2+d}$ which appear $4$ times. The $2d-1+1+2d-1=4d-1$ non-zero columns of $M$ are as follows. 

\begin{align}\label{leftpart}
\begin{split}
m_{kd+1}&=-\zeta^ke_{k+1}+e_{2d+k+2},~k=0,\ldots, d-2\\
m_{(d-1)d+1}&=-\frac{d}{\zeta}e_1-\zeta^{d-1}e_d+de_{2d+2}+e_{3d+1}\\
m_{(d-1)d+1+i}&=-\frac{d}{\zeta}e_{i+1}+de_{2d+2+i},~i=1,\ldots, d-1.
\end{split}
\end{align}
This covers the first $d(d-1)+d-1$ columns of the matrix. The central column of $M$ is 
$$m_{d^2+1}=-\frac{1}{\zeta}(d+1)e_{d+1}+(d+1)e_{3d+2}.$$
The following non-zero columns are 
\begin{align}\label{rightpart}
\begin{split}
m_{d^2+1+j}&=de_{d+1+j}-de_{3d+2+j},~j=1,\ldots,d-1\\
m_{d(d+1)+1}&= \zeta e_{d+2}+d e_{2d+1}-e_{3d+3}-de_{4d+2}\\
m_{d(d+1)+1+sd}&=\zeta^{s+1} e_{d+s+2}- e_{3d+3+s}, ~s=1,\ldots, d-1.
\end{split}
\end{align}

From \eqref{leftpart} and \eqref{rightpart} we infer  
\begin{align}
\begin{split}
\<m_{kd+1},m_{(d-1)d+k+1}\>&=\<e_{k+1},e_{2d+k+2}\>,~k=1,\ldots,d-2\\
\<m_{d^2+k+1},m_{d(d+1)+1+(k-1)d}\>&=\<e_{d+k+1},e_{3d+k+2}\>, ~k=2,\ldots,d-1,
\end{split}
\end{align}
and the space generated by the above column vectors $V$ has dimension $2d-4+2d-4$. 
Notice that we have not yet considered the column vectors 
\begin{align}
\begin{split}
m_1&=-e_1+e_{2d+2}\\ 
m_{d^2}&=-\frac{d}{\zeta}e_{d}+de_{3d+1} \\
m_{d^2+2}&=de_{d+2}-de_{3d+3}\\ 
m_{2d^2+1}&=\frac{1}{\zeta}e_{2d+1}-e_{4d+2}\\ 
m_{d^2+1}&=-\frac{1}{\zeta}(d+1)e_{d+1}+(d+1)e_{3d+2}. 
\end{split}
\end{align}
which are mutually orthogonal and so generate a $5$-dimensional space $W$. Further $W$ is contained in the orthogonal complement of $V$ so that $\dim V+W = 4d-3$. 
Finally, look at the vectors 
\begin{align}
\begin{split}
m_{(d-1)d+1}&=-\frac{d}{\zeta}e_1-\zeta^{d-1}e_d+de_{2d+2}+e_{3d+1}\\
m_{d(d+1)+1}&= \zeta e_{d+2}+d e_{2d+1}-e_{3d+3}-de_{4d+2}
\end{split}
\end{align}
They are clearly linearly independent so generate a $2$-dimensional $U$. We can easily see that they don't belong in $V+W$. 
Indeed, the only vector in $V+W$ involving $e_1$ is $m_1$, but notice that it also involves $e_{2d+2}$ and the coefficients don't match that of $m_{(d-1)d+1}$. Similarly, the only vector involving $e_{d+2}$ in $V+W$ is $m_{d^2+2}$, which also involves $e_{3d+3}$ in a way that does not match $m_{d(d+1)+1}$. 

Thus the rank of our matrix is at least $\dim V+W+U = 4d-1$ and the lemma follows. 
\end{proof}

\begin{remark}
It is necessary to take a primitive $(d+1)$-th root of unity for the dimension in Lemma \ref{linear algebra} to be $4d-1$. Clearly, the dimension is smaller for $\zeta=1$. If on the other hand we chose $\zeta$ with $\zeta^d=\zeta^k$ for some $k\in\{0,\ldots, d-2\}$, then (at least) two non-zero columns of the matrix $M$ in the proof of Lemma \ref{linear algebra} are related. For instance we have $$m_{(d-1)d+k+1}=dm_{kd+1}=-d\zeta^de_{k+1}+de_{2d+2+k}.$$ 
\end{remark}

\subsection{Proof of Theorem \ref{monomial theorem}}
With Theorem \ref{rigidity of special pair} in hand, we can now complete the proof of Theorem \ref{monomial theorem}.  

Enumerate the roots of unity as $\{\xi_i\}_{i\geq 1}$, in any order.  For pairs $(f,g) \in S_d$ near $(f_0, g_0)$, let $P_i$ (respectively, $Q_i$) denote a parameterization of the preperiodic point of $f$ (respectively, $g$) that agrees with $\xi_i$ at $(f_0, g_0)$.  By stability, each of these $P_i$ and $Q_i$ is well defined and smooth at $(f_0, g_0)$. Note also that the subvariety of $S_d$ defined by $V_1 = \{P_1 = Q_1\}$ cannot be all of $S_d$, because there exist pairs $(f,g)$ with all of their preperiodic points disjoint. Thus the codimension of $V_1$ is 1.  Now consider the subvarieties $V_i = \{P_i = Q_i\}$ and their intersections with $V_1$, for all $i$. If all of them coincide with $V_1$ near $(f_0, g_0)$, then we would have $\Preper(f) \cap J(f) = \Preper(g) \cap J(g)$ persistently along $V_1$ near $(f_0, g_0)$.   This contradicts Theorem \ref{rigidity of special pair}.  Therefore, there exists an index $i$ so that $V_1 \cap V_i$ has codimension 2 near $(f_0, g_0)$.  Continuing inductively in this way, we find a $(4d-1)$-tuple of roots of unity that form a rigid $(4d-1)$-repeller at $(f_0, g_0)$. 

This completes the proof of Theorem \ref{monomial theorem}.
\qed

\bigskip
\section{Latt\`es maps}\label{Lattes maps}
\label{BFT section}

In this final section, we prove Theorem \ref{Lattes theorem intro}, restated here as Theorem \ref{Lattes theorem}.   

Let $\cL$ denote the Legendre family of flexible Latt\`es maps in degree 4, defined by 
	$$f_t(z) = \frac{(z^2-t)^2}{4z(z-1)(z-t)}$$ 
for $t \in \C\setminus\{0,1\}$.  This $f_t$ is the quotient of the multiplication-by-2 endomorphism of the Legendre elliptic curve 
	$$E_t = \{y^2 = x(x-1)(x-t)\}$$
via the projection $(x,y)\mapsto x$.  The preperiodic points of $f_t$ coincide with the projection of the torsion points of $E_t$.  

\begin{theorem} \label{Lattes theorem}
For each degree $d\geq 2$, there exists a uniform bound $M_d$ so that either
	$$|\Preper(f) \cap \Preper(g)| \leq M_d \quad\mbox{or} \quad \Preper(f) = \Preper(g)$$
for all pairs $(f,g)$ with $f \in \cL$ and $g \in \Rat_d$.  
\end{theorem}

\begin{cor}\label{Bft:cor}  
There exists a constant $M>0$ such that for every pair of elliptic curves $E_1$ and $E_2$ over $\C$, equipped with degree-two projections $\pi_i: E_i \to \P^1$ ramified at the 2-torsion points $E_i[2]$, we have 
	$$|\pi_1(E^{\mathrm{tors}}_1)\cap \pi_2(E^{\mathrm{tors}}_2)|\le M,$$
if and only if $\pi_1(E_1[2]) \not= \pi_2(E_2[2])$.    
\end{cor}

\begin{cor}  \label{roots of unity cor}
Let $\mu_{\infty}\subset \C$ denote the set of roots of unity. 
There exists a constant $B>0$ such that 
\begin{align*}
|\pi(E^{\mathrm{tors}})\cap \mu_{\infty}|\le B,
\end{align*}
for every elliptic curve $E$ defined over $\C$ and any degree-2 projection $\pi:E\to \P^1$ ramified at the 2-torsion points of $E$.
\end{cor}  

\subsection{Non-isotriviality} \label{ni}
Fix a degree $d \geq 2$.  Let $\cL(d) := \cL \times \Rat_d$.  Consider the map $\cL(d) \to \Rat_{d^2} \times \Rat_{d^2}$ which sends a pair $(f_t, g)$ to the pair $(f_{d,t}, g^2)$, where $f_{d,t}$ is the quotient of the multiplication-by-$d$ endomorphism on the elliptic curve $E_t$ and $g^2$ is the second iterate of $g$.  Note that $\Preper(f_t) = \Preper(f_{d,t})$ for all $t\in \C\setminus\{0,1\}$ and $\Preper(g) = \Preper(g^2)$.

\begin{prop}  \label{nonisotrivial L}
The induced map $\cL(d) \to (\Rat_{d^2} \times \Rat_{d^2})/\Aut\P^1$ to the moduli space of pairs is finite-to-one, so $S = \cL(d)$ parametrizes a maximally non-isotrivial family of pairs of maps of degree $d^2$.  
\end{prop}

\begin{proof}
Two maps $f_{d, t_1}$ and $f_{d,t_2}$ are conjugate if and only if the elliptic curves $E_{t_1}$ and $E_{t_2}$ are isomorphic.  Moreover, the map $\Rat_d \to \Rat_{d^2}$ defined by iteration is finite, because it is proper between affine varieties; see, e.g., \cite[Corollary 0.3]{D:measures}. 
\end{proof}

\subsection{Non-degeneracy}  
Fix a degree $d \geq 2$.  Let $\cL(d) := \cL \times \Rat_d$ parameterize the family of pairs of degree $d^2$ as in \S\ref{ni}.  

\begin{prop}  \label{nd prop}
Let $S \to \cL(d)$ be a finite map from a smooth, irreducible quasiprojective algebraic variety $S$ of dimension $m \geq 1$, defined over $\C$, and let $\Phi = (f,g)$ be the associated algebraic family of pairs of degree $d^2$ over $S$.  Then either there exists a rigid $m$-repeller at some point $s_0 \in S(\C)$ or $\Preper(f_s) = \Preper(g_s)$ for all $s \in S(\C)$.
\end{prop}

\begin{proof}

First note that $\Phi = (f,g)$ is maximally non-isotrivial, because the map to $\cL(d)$ is finite.  Let $m = \dim S$.  Assume that we do not have $\Preper(f_s) = \Preper(g_s)$ for all $s \in S(\C)$.  To show the existence of a rigid $m$-repeller at a parameter $s_0 \in S(\C)$, we repeat the arguments in the proof of Theorem \ref{density result} to build common preperiodic points. While Theorem \ref{density result} shows there are at least $m$ common preperiodic points for a Zariski dense set of pairs $(f_s, g_s)$ in $S(\C)$, it does not give control over whether they are repellers nor whether they will be rigid.  We use the fact that $f$ is Latt\`es to provide this.

For simplicity, we first give the proof assuming that both $f$ and $g$ are Latt\`es maps throughout $S$.  In particular, the periodic points of $f_s$ and $g_s$ are all repelling, for all $s \in S(\C)$.  Let $P_1$ denote a hypersurface in $S \times \P^1$ parameterizing a periodic point for $f_s$, chosen so that it is not persistently preperiodic for $g$.  If no such hypersurface exists, then we deduce that $\Preper(f) = \Preper(g)$ throughout $S$ (as a consequence of Theorem \ref{LPYZ}) and we are done.  So we assume that we  have such a $P_1$.  As in the proof of Theorem \ref{density result}, we will apply Theorem \ref{stable point} to the pair $(g, P_1)$ over a branched cover $S' \to S$ where $P_1$ may be viewed as the graph of a point in $S'\times\P^1$.  If the pair $(g,P_1)$ is isotrivial over $S'$, then we replace $P_1$ with another periodic point for $f$.  If the pair $(g,P_1)$ is isotrivial for all periodic points $P_1$ of a given large period $>2d^2$, then interpolation (as in the proof of Theorem \ref{density result}) implies that the pair $(f,g)$ must be isotrivial, contradicting our assumption.  We conclude from Theorem \ref{stable point} that there exists a parameter $s_1 \in S'(\C)$ at which $P_1$ is preperiodic to a repelling point for $g_{s_1}$.  We then let $P_1' \subset P_1$ be the codimension 1 subvariety containing $(s_1, P_1(s_1))$ along which both $f$ and $g$ are persistently preperiodic.  Then $P_1'$ projects to a subvariety $S_1 \subset S'$ of codimension 1.   We now repeat the argument with another periodic point $P_2$ for $f$ over $S_1$, distinct from $P_1'$.  We continue inductively, using the fact that $(f,g)$ is maximally non-isotrivial, to find $m$ distinct common preperiodic points at some parameter $s_0\in S(\C)$ that form a rigid $m$-repeller.

Now we assume that $g$ is not everywhere Latt\`es.  As the Latt\`es pairs $(f,g)$ form a proper subvariety of $S$, we replace $S$ with a Zariski open subset so that $g$ is not a Latt\`es map for any $s \in S(\C)$.  

Again let $m = \dim S$.  Let $P_1$ denote a hypersurface in $S \times \P^1$ parameterizing a periodic point for $f$, chosen so that it is not persistently preperiodic for $g$.  If no such curve exists, then, as above, we have $\Preper(f) = \Preper(g)$ and we are done.  So, we may assume that $P_1$ exists.  Again as in the proof of Theorem \ref{density result}, we pass to a branched cover $S'\to S$ so that $P_1$ may be viewed as the graph of a point in $S'\times \P^1$.  If the pair $(g, P_1)$ is isotrivial, and if this holds for all choices of $P_1$, then the pair $(f,g)$ is isotrivial by an interpolation argument, and we have a contradiction.  So from Theorem \ref{stable point}, there exists a parameter $s_1\in S'(\C)$ where the point $P_1$ is preperiodic to a repelling cycle for $g$, but not persistently so.  As above, we let $P_1' \subset P_1$ be the codimension 1 subvariety containing $(s_1, P_1(s_1))$ along which both $f$ and $g$ are persistently preperiodic.  Then $P_1'$ projects to a subvariety $S_1 \subset S'$ of codimension 1.

Since $P_1'$ is preperiodic to a repelling cycle for $g_{s_1}$, there is an open neighborhood $U_1$ of $s_1$ in $S_1$ on which that cycle remains repelling for $g_s$.  The density of stability implies that we can find an open $U_1' \subset U_1$ on which the family $g_s$ is stable for $s \in U_1'$, so its Julia set (and in particular, including all repelling periodic points) is moving holomorphically.  As the family $f$ is stable over all of $S_1$, we also have a holomorphic motion of its preperiodic points, and these are dense in $\P^1$ over every $s\in S_1$.  Thus there are only two possibilities:  either there is a codimension-1 intersection of one of the pre-repelling points of $g$ with a preperiodic point $\not= P_1'$ of $f$ at some parameter in $U_1'$, or each of the preperiodic points of $g$ in its Julia set becomes a leaf of the holomorphic motion of the Julia set of $f$.  In the latter case, by analytic continuation, this shared holomorphic motion must persist over all of $S_1$.  But then the algebraic family of maps $g$ must itself be stable on all of $S_1$, as there would be no collisions between the distinct periodic points; see \S\ref{stable subsection}.  It follows from McMullen's theorem \cite{McMullen:families} that $g$ is also a family of Latt\`es maps, which is a contradiction.  Thus we can find a parameter $s_2 \in U_1'$ so that preperiodic points of $f$ and $g$ in their Julia sets intersect in a codimension 1 subvariety of $U_1' \times \P^1$.  

The proof is completed by induction on dimension.
\end{proof}

\subsection{Proof of Theorem \ref{Lattes theorem} and its corollaries}  

\begin{proof}[Proof of Theorem \ref{Lattes theorem}]
Fix a degree $d \geq 2$.  As in \S\ref{ni}, we let $\cL(d) = \cL \times \Rat_d$ and let $\Phi = (f,g)$ denote the algebraic family of maps of degree $d^2$ parameterized by $\cL(d)$.  This $\Phi$ is maximally non-isotrivial, by Proposition \ref{nonisotrivial L}.  Note that $\dim \cL(d) = 2d+2$.  

Since there exist $g \in \Rat_d$ with Julia sets $J(g) \not= \P^1$, we do not have $\Preper(f_s) = \Preper(g_s)$ for all $s \in \cL(d)$.  Thus, from Proposition \ref{nd prop}, there exists a rigid $(2d+2)$-repeller at some parameter $s_0 \in \cL(d)$.  It follows from Corollary \ref{our measure} that the pairwise-bifurcation measure $\mu_{\Phi, \Delta}$ is nonzero on $\cL(d)$.  Then from Theorem \ref{uniform bound}, there exists a Zariski-closed subvariety $V_1 \subset \cL(d)$ of codimension 1, defined over $\Qbar$, and a constant $M_1$ so that 
	$$\# \Preper(f_s) \cap \Preper(g_s) \leq M_1$$
for all $s \in (\cL(d) \setminus V_1)(\C)$.  

We then repeat these arguments on each irreducible component $V_1'$ of $V_1$.  Either $\Preper(f_s) = \Preper(g_s)$ for all $s \in V_1'(\C)$ or there is a subvariety $V_2' \subset V_1'$ of codimension 1, defined over $\Qbar$, and a constant $M_2$ so that $\# \Preper(f_s) \cap \Preper(g_s) \leq M_2$ for all $s \in V_1' \setminus V_2')(\C)$.  We let $V_2$ be the union over all of the $V_2'$.  It has codimension 2 in $\cL(d)$. Induction on dimension completes the proof.
\end{proof}

\begin{proof}[Proof of Corollary \ref{Bft:cor}]
We apply Theorem \ref{Lattes theorem} to the algebraic family of pairs $\Phi = (f,g)$ for $f \in \cL$ and $g$ the family of all conjugates of maps in $\cL$.  More precisely, we consider the subvariety $V \subset \Rat_4$ of all maps that are conjugate to elements of $\cL$ by M\"obius transformations, and let $S = \cL \times V$.  Theorem \ref{Lattes theorem} then implies that there is a constant $M$ so that either
	$$|\Preper(f_s) \cap \Preper(g_s)| \leq M \quad\mbox{or} \quad \Preper(f_s) = \Preper(g_s)$$
for all $s \in S(\C)$.  For any pair of elliptic curves $E_1$ and $E_2$ over $\C$, equipped with their degree-two projections $\pi_i: E_i\to \P^1$, there exists a M\"obius transformation $A \in \Aut\P^1$ and an $s \in S(\C)$ so that 
	$$A(\pi_1(E_1^{\mathrm{tors}})) = \Preper(f_s) \quad \mbox{and} \quad A(\pi_2(E_2^{\mathrm{tors}})) = \Preper(g_s).$$
Observing also that $\pi_1(E_1[2]) = \pi_2(E_2[2])$ if and only if $\pi_1(E_1^{\mathrm{tors}}) = \pi_2(E_2^{\mathrm{tors}})$ if and only $\Preper(f_s) = \Preper(g_s)$, the proof is complete.
\end{proof}

\begin{proof}[Proof of Corollary \ref{roots of unity cor}]
We apply Theorem \ref{Lattes theorem} to the algebraic family of pairs $\Phi = (f,g)$ for $f \in \cL$ and $g$ the family of all conjugates of the map $g_0(z) = z^2$.  More precisely, let $V \subset \Rat_2$ be the $\Aut\P^1$-orbit of $g_0$.  Note that $\Preper(g_0) \supset \mu_\infty$, the set of all roots of unity.  Let $S = \cL \times V$.  For any elliptic curve $E$ over $\C$, equipped with its degree-two projection $\pi: E \to \P^1$, there exists a M\"obius transformation $A \in \Aut\P^1$ and an $s \in S(\C)$ so that 
	$$A(\pi(E^{\mathrm{tors}})) = \Preper(f_s) \quad \mbox{and} \quad A(\mu_\infty) \subset \Preper(g_s).$$
Note that we cannot have $\Preper(f_s) = \Preper(g_s)$ for any $s \in S(\C)$, because the Julia set of $g_s$ is a circle while $J(f_s) = \P^1$.  We apply Theorem \ref{Lattes theorem} to complete the proof.
\end{proof}


\bigskip \bigskip
\def\cprime{$'$}

\bigskip\bigskip

\end{document}